\documentclass[10pt]{amsart}
\usepackage{enumitem}   
\usepackage{amsmath}
\usepackage{amssymb}
\usepackage{amsthm}
\usepackage{comment}
\usepackage{hyperref}
\usepackage{mathrsfs} 
\usepackage{breqn}
\usepackage{xcolor}
\usepackage{graphicx}
\usepackage{stmaryrd}
\usepackage{geometry}
\let\vec\mathbf

\usepackage{subcaption}

\newtheorem{thm}{Theorem}[section]
\newtheorem{ex}[thm]{Example}
\newtheorem{prop}[thm]{Proposition}

\newtheorem{dfn}[thm]{Definition}
\newtheorem{rmk}[thm]{Remark}

\newcommand{\UU}{\mathfrak{U}}

\newcommand{\FF}{\mathfrak{F}}

\newcommand{\TT}{\mathscr{T}}

\newcommand{\rr}{\vec{r}_t}

\newcommand{\T}{\mathcal{O}}
\newcommand{\R}{\mathbb{R}}

\newcommand{\F}{\mathbb{F}}
\newcommand{\E}{\mathbb{E}}
\newcommand{\p}{\mathbb{P}}

\newcommand{\N}{\mathbb{N}}

\usepackage{authblk}

\numberwithin{equation}{section}

\newcommand{\dd}{\mathrm{d}}
\newcommand{\dx}{\,\mathrm{d}x}
\newcommand{\dt}{\,\mathrm{d}t}

\newcommand{\bfr}{\mathbf r}
\newcommand{\bfU}{\mathbf U}

\newcommand{\EE}{\mathscr E}

\newcommand{\ee}{\mathrm e}
\newcommand{\bn}{\mathbf n}
\newcommand{\VV}{\mathcal V}
\newcommand{\bU}{\mathbf{U}}
\newcommand{\bu}{\mathbf u}
\newcommand{\bh}{\mathbf h}
\newcommand{\bm}{\mathbf m}
\newcommand{\bV}{\mathbf V}
\newcommand{\diff}{\mathrm{d}}
\newcommand{\bx}{\mathbf x}
\newcommand{\bg}{\mathbf{g}}
\newcommand{\diag}{\mathrm{diag}}
\newcommand{\IdxM}{\mat{\mathbb I_4}}
\newcommand{\bvu}{\underline{\bu}}
\newcommand{\bvf}{\underline{\bf}}
\newcommand{\bfnum}{\mathbf f^{\mathrm{num}}}

\usepackage[foot]{amsaddr}
\newcommand{\armean}[1]{\{\{#1\}\}}
\newcommand{\logmean}[1]{\{\{#1\}\}_{\log}}
\newcommand{\jump}[1]{\llbracket#1\rrbracket}
 \newcommand{\mat}[1]{\underline{\underline{#1}}}
 \newcommand{\norm}[1]{\|#1\|}

\title{Discontinuous Galerkin methods for the complete stochastic Euler equations}
\author{Dominic Breit}
\author{Thamsanqa Castern Moyo}
\author{Philipp  \"Offner}
\address{TU Clausthal, Institute of Mathematics, Erzstra\ss e 1, 38678 Clausthal-Zellerfeld, Germany.}
\email{dominic.breit@tu-clausthal.de}
\email{thamsanqa.castern.moyo@tu-clausthal.de}
\email{mail@philippoeffner.de}

\begin{document}
\maketitle
\begin{abstract}

In recent years, stochastic effects have become increasingly relevant for describing fluid behaviour, particularly in the context of turbulence. The most important model for inviscid fluids
in computational fluid dynamics are the Euler equations of gas dynamics which we focus on in this paper. 
To take stochastic effects into account, we  incorporate  a stochastic forcing term in the momentum equation of the Euler system. 
To solve the extended system, we apply an entropy dissipative discontinuous Galerkin spectral element method including the Finite Volume setting,  adjust it to the stochastic Euler equations and analyze its convergence properties.
Our analysis is grounded in the concept of \textit{dissipative martingale solutions}, as recently introduced by Moyo (J. Diff. Equ. 365, 408-464, 2023). Assuming no vacuum formation and bounded total energy,  we proof that our scheme converges in law to a dissipative martingale solution. During the lifespan of a pathwise strong solution, we achieve convergence of at least order 1/2, measured by the expected $L^1$ norm of the relative energy. The results built a counterpart of corresponding results in the deterministic case.
In numerical simulations, we show the robustness of our scheme, visualise different stochastic realizations and analyze our theoretical findings.

\end{abstract}
\section{Introduction}

The Euler equations of gas dynamics are one of -if not the- most studied system for describing inviscid, compressible fluid flow. 
They contain the three conservation laws  of mass, momentum, and energy and serve as foundational tools for understanding and predicting fluid behavior in applications like aerodynamics, astrophysics, and meteorology. The system has been thoroughly investigated from an analytical and numerical perspective. Numerically, many different schemes have been developed to solve the Euler equations in a highly efficient way. The aim is always to obtain the most accurate result for real-world problems. From a modeling perspective, the inclusion of uncertainty and/or to  describe 
turbulent effects, researchers have recently started to introduce stochastic forcing into the Euler equations (an other models as well) to describe more realistic scenarios. \\
In this paper, we follow this trend and study 
the complete Euler equations describing the motion of a general inviscid, compressible and heat-dependent fluid subject to a stochastic perturbation in the momentum equation. In the variables density $\varrho$, momentum $\vec m$ and specific internal energy $e$, the considered system reads then as 
\begin{align}\label{Euler}
\begin{aligned}
\dd \varrho + \mathrm{div}_{x}(\vec m)\, \dd t &=0 \quad \text{in}\, Q,\\
\dd \vec m+ \mathrm{div}_x\Big(\frac{\vec m \otimes \vec m}{\varrho}\Big)\, \dd t + \nabla_x p\,\dd t&=\varrho \phi \,\dd W\quad \text{in}\, Q,\\
\dd (\varrho e) + \mathrm{div}_x\left(\varrho e\frac{\vec m}{\varrho}\right)\dd t &=-p\,\mathrm{div}\Big(\frac{\vec m}{\varrho}\Big)\dd t \quad \text{in}\, Q,\end{aligned}
\end{align}
where $Q=(0,T)\times\mathcal O$ for some $T>0$ and $\mathcal O\subset\R^n$, $n=1,2,3$, bounded.
Here,  $W$ denotes a (cylindrical) Wiener process on a filtered probability space $(\Omega, \FF, (\FF_t)_{t\geq 0},\p)$ and $\phi$ the noise coefficient (a Hilbert-Schmidt operator\footnote{We refer to Section \ref{sec:prob} for the precise probability framework. }). 
The key idea behind adding the stochastic perturbation to the deterministic Euler system is to take physical, empirical or numerical uncertainties and thermodynamical fluctuations into account. Additionally, it is also commonly used to model turbulent effects, see \cite{Bi,MiRo} and the references therein.   
However, it is well-known that the deterministic counterpart of \eqref{Euler} is well-posed only locally in time and solutions can develop shocks \cite{Sm}. Recently, also the precise behaviour of singularities could be described \cite{BSV}. Due to the technique of convex integration, there has been, and continues to be, significant interest in constructing scenarios where the deterministic Euler equations yield multiple entropy solutions from the same initial data. See, for example, \cite{ChDeO,DeSz} for the isentropic Euler system and \cite{FKKM} for the complete Euler equations.
This is not only valid in the deterministic case but also if stochastic perturbations are added, cf.  \cite{BFHcv} for the isentropic Euler equations and in \cite{ChFeFl} for the complete Euler equations.
The only chance to obtain globally defined solutions which obey the laws of thermodynamics in a stable way is the concept of measure-valued solutions. 
 It was already introduced (in the deterministic case for incompressible fluids) in the 80's \cite{DiP,DiMa} but attracted a lot of attention again recently due to the weak-strong uniqueness principle available for dissipative solutions \cite{HoFeBr,Bren,FEJB,KrZa,Neu} and to demonstrate convergence of structure preserving methods \cite{feireisl2020convergence,feireisl2021numerical} within this framework.   
Weak-strong uniqueness  results were  recently extended to the stochastic case  by the second author in \cite{Mo} where he showed the existence of a dissipative martingale solution to \eqref{Euler}. Such a solution is weak in the probabilistic sense which means the probability space is not a priori given but becomes an integral part of the solution. \\
Meanwhile from the numerical perspective, there has been a long history on solving the deterministic Euler equations. 
An instrumental method in this context is the approach pioneered by Godunov \cite{godunov1959finite}, which can be conceptualised as a finite volume (FV) technique that tackles exact Riemann problems at each cell interface. Godunov's method has served as the foundation for numerous extensions and has been employed as the groundwork for the development of higher-order methods, exemplified in works such as \cite{harten1983upstream,roe1981approximate,toro1997riemann}. Despite its historical significance, Godunov's method continues to attract attention, with recent studies showcasing its convergence through dissipative weak (DW) solutions of the Euler equations \cite{LuYu}. Additionally, novel error estimates for the multidimensional compressible Euler system have been introduced in \cite{LuSh}. Notably, the results of \cite{LuYu} build upon convergence outcomes for several first-order FV schemes, including the (local) Lax-Friedrichs method, as discussed in \cite{feireisl2020convergence}. For a comprehensive overview of recent convergence results, we direct readers to \cite{feireisl2021numerical} for further details.
In parallel with low-order methods, contemporary applications involve high-order methods for solving hyperbolic conservation laws, particularly the Euler equations. Among these methods, the Discontinuous Galerkin (DG) method stands out as one of the most favoured within the hyperbolic community. Originating from the work of Reed and Hill in 1973, initially devised for solving the hyperbolic neutron transport equation in nuclear reactors \cite{ReHi}, DG has evolved over the years. Mathematical foundations were solidified in the 90s \cite{CoKa, HeWa}, and more recent advancements have focused on ensuring structure-preserving (or property-preserving) properties, as pointed out in \cite{KuHa, oeff2023} and references therein. In the context of the Euler equation, recent convergence results via dissipative weak solutions have been demonstrated for DG in \cite{LuOe}. This marked a significant milestone as the first convergence result for a high-order method within the DW framework. Later one these results were extended to other structure-preserving Finite Element (FE) based schemes  in \cite{AbLu, KuLu}.
All convergence results using the DW framework are made on the assumption that the numerical solution $(\varrho^h,\vec m^h,\mathcal E^h)$, usually formulated in terms of the total energy $\mathcal E^h$ rather than the internal energy, satisfies for some $K>0$
 \begin{align}\label{ass:mainA}
\inf_{t\in[0,T],x\in\T}\varrho^h(t,x)\geq \frac{1}{K},\quad \sup_{t\in[0,T],x\in\T}\mathcal E^h(t,x)\leq K,
\end{align}
uniformly in $h$. Here $h$ is the corresponding grid parameter and the superscript denotes that the quantity is the numerical approximation, cf.  \cite{AbLu,feireisl2021numerical, KuLu, LuYu}. The physical interpretation of \eqref{ass:mainA} is that the density is not reaching vacuum and we have no blow up in terms of energy (and so also the speed is bounded). 
In the stochastic case\footnote{Here, we refer that we consider stochastic forcing inside the equations not uncertain initial/boundary data or uncertainty inside the flux where other techniques like Monte-Carlo simulations, general Polynomial chaos or stochastic collocation methods are used and heavily considered \cite{AbMi}.}, the only available result -up to the authors knowledge- is the paper \cite{ChKo}, where a finite volume scheme is considered for the barotropic Euler equations. The authors from \cite{ChKo} work under assumption \eqref{ass:mainA}. Unfortunately, this is not realistic for stochastic PDEs. Due to the Gaussian character inherited from the driving Wiener process solutions are in general not bounded in probability, not even locally in time. This is even the case for linear problems with smooth data which are globally well-posed, see, for instance, \cite{DP}. \\
Due to this we replace \eqref{ass:mainA} by the following assumption which is on the continuous level certainly satisfied by a strong solution (see Definition \ref{def:compstrong}) and thus seems much more plausible. We suppose that there is a stopping time $\mathfrak t$ with $\p(\mathfrak t>0)=1$ and a deterministic constant $K>0$ such that $\p$-a.s.
\begin{align}\label{ass:mainB}
\inf_{t\in[0,\mathfrak t],x\in\T}\varrho^h(t,x)\geq \frac{1}{K},\quad \sup_{t\in[0,\mathfrak t],x\in\T}\mathcal E^h(t,x)\leq K,
\end{align}
uniformly in $h$. Heuristically speaking, we assume that \eqref{ass:mainB} is only fulfilled up to some time $\mathfrak t$ which exists with probability  $\p(\mathfrak t>0)=1$.
Let us explain some heuristics behind this assumption. If the velocity gradient is bounded in space-time the standard maximum principle for the continuity equation \eqref{Euler}$_1$ implies strict positive of the density. However such a bound on the velocity field can never be uniform in probability (as the stochastic forcing is not uniformly bounded). Hence, for any given deterministic time $T>0$, any positiv deterministic bound can be deceeded with a positive probability. So, out of all possible paths the solution can attain some of them will not satisfy \eqref{ass:mainA}. One can consequently either neglect certain paths or adapt the time horizon. For any given path there is a (possibly small) time for which the movement of the noise is small. Thus \eqref{ass:mainB} holds provided the solution is sufficiently smooth in space.\\
 Based on this assumption we are able to prove the consistency and convergence\footnote{Convergence means in such context: Convergence in law up to a subsequence and it is based on the stochastic compactness method employing Jakubowski's extension of the Skorokhod representation theorem from \cite{Jaku} (due to the defect measures, which appear in the nonlinear term when passing to the limit in the discrete solution, the classical version for Polish spaces does not apply). A main feature is that we have to include the stopping from \eqref{ass:mainA} in the path space which creates several technical difficulties and will be explained in all details in the following chapters.} between a  DG discretization (including the FV case) of \eqref{Euler} and the dissipative martingale solution from \cite{Mo}. The convergence result  can be improved in the life-space of a strong solution to \eqref{Euler} which exists locally in time up to a stopping time
$\mathfrak s$ (which we expect to be significantly smaller than $\mathfrak t$ in \eqref{ass:mainB}), cf. Definition \ref{def:compstrong}, and is defined on a given stochastic basis. 
In this case we obtain convergence of order $1/2$, where the error is measured as the expectation of the $L^1$ norm of the relative energy between discrete and continuous variables evaluated at the final time. 
%
The result is based on the relative entropy of the discrete solution and provides a stochastic counterpart of the deterministic case. 
Therefore, the paper is organized as follows:
In Section \ref{sec:cr}, we introduce the notation, the mathematical framework and the discontinuous Galerkin (DG) method under consideration. We give a focus on the stochastic part and how we are dealing with it.
 Then, in Section \ref{se_concept} we repeat the definition of dissipative measure-valued martingale solutions from \cite{Mo}, strong solutions in the probabilistic and PDE sense, and formulate the main results about consistency and convergence results. The chapter should give a detailed introduction inside the topic and provide the main results in this context.  In Section \ref{mvsproof}, the consistency results are proven in all technical details where in Section \ref{se_Convergence} the focus is on the convergence result. In addition, in Section  \ref{sec_numerical} we show -up to the authors knowledge- for the first time numerical results using the extended Euler system \eqref{Euler}. Furthermore, we verify our theoretical findings and discuss problems which may rise as well inside the numerical experiments.
A conclusion and an outlook in Section \ref{se_outlock} finishes this paper.

\section{Mathematical framework}\label{Asec}
In this section, we introduce the stochastic Euler system \eqref{Euler} and describe different formulations for it. On the continuous level (for sufficiently smooth solutions) they are equivalent, but from a numerical perspective one is more  favorable than the other as already can be seen in the deterministic case in \cite{abgrall2018,abgrall2010}.  Next, we provide a brief introduction to the probability framework; for further details, see \cite[Chapter 2]{FrBrHo}. Finally, we introduce the discontinuous Galerkin (DG) method under consideration the well-known DG spectral element method using summation-by-parts operators and a flux differencing approach. The method is well investigated and well-used applied for the deterministic Euler equations, cf.  \cite{chen2017,He,LuOe,renac_2019} and references therein. We adjust it to the stochastic Euler system.
\subsection{Constitutive relations}\label{sec:cr}
The fluid model is described by means of three basic state variables: the mass density $\varrho=\varrho(t,x)$, the momentum $\vec m = \vec m(t,x)$, and the (absolute) temperature $\vartheta=\vartheta(t,x)$, where $t$ is the time, $x$ is the space variable (Eulerian coordinate system). The time evolution of the fluid flow is governed by the system of partial differential equations given by
\begin{align}\label{EulerB}
\begin{aligned}
\dd \varrho + \mathrm{div}_{x}\vec m\, \dd t &=0 \quad \text{in}\, Q,\\
\dd \vec m+ \mathrm{div}_x\Big(\frac{\vec m \otimes \vec m}{\varrho}\Big)\, \dd t + \nabla_x p\,\dd t&=\varrho \phi \,\dd W\quad \text{in}\, Q,\\
\dd e + \mathrm{div}_x\left(e\frac{\vec m}{\varrho}\right)\dd t &=-p\,\mathrm{div}\Big(\frac{\vec m}{\varrho}\Big)\dd t \quad \text{in}\, Q
,\end{aligned}
\end{align}
describing: the balance of mass, momentum and internal energy, respectively. Here, $p(\varrho,\vartheta)$ denotes pressure and $e=e(\varrho,\vartheta)$ is the internal energy. 
In the theoretical part, we focus on the two-dimensional Euler system \eqref{EulerB}. However, the same arguments apply in both one and three dimensions\footnote{In Section \ref{sec_numerical}, we demonstrate also some results in one dimension for completeness.}. The total energy is the sum of the kinetic and internal energy, respectively, that is
\begin{align*}
\mathcal E=\mathcal E(\varrho,\vartheta)=\frac{1}{2}\frac{|\vec m|^2}{\varrho}+\varrho e(\varrho,\vartheta).
\end{align*}
Applying formerly It\^{o}'s formula to the process $t\mapsto \frac{1}{2}\frac{|\vec m|^2}{\varrho}$, the internal energy equation 
in \eqref{EulerB} can equivalently be written in terms of the total energy  as\footnote{In the case of a cylindrical Wiener process, see \eqref{Dwiener} below, one has $\|\phi\|^2_{\ell_2}:=\sum_{k\in\N}|\phi e_k|^2$. Note that this is a function depending on space.}
\begin{equation}\label{eq_total}
\dd \mathcal E + \mathrm{div}_x\left(\left(\mathcal E+p\right)\frac{\vec m}{\varrho}\right)\dd t =\frac{1}{2}\varrho\|\phi\|_{\ell_2}^2 \,\dd t+ \phi\cdot \vec m\,\dd W.
\end{equation}
 For completeness, the system (\ref{EulerB}) is supplemented by a set of constitutive relations characterising the physical principles of a compressible inviscid fluid. In particular, we assume that the pressure $p(\varrho,\vartheta)$ and the internal energy $e =e(\varrho,\vartheta)$ satisfy the caloric equation of state
\begin{equation}\label{caloric}
    p=(\gamma-1)\varrho e, 
\end{equation}
where $\gamma>1$ is the adiabatic constant. In addition, we suppose that the absolute temperature $\vartheta$ satisfies the Boyle-Mariotte thermal equation of state:
\begin{equation}\label{boyle}
    p =\varrho\vartheta \quad \mathrm{yielding}\quad e= c_v\vartheta, c_v =\frac{1}{\gamma-1}.
\end{equation}
Finally, we suppose that the pressure $p=p(\varrho, \vartheta)$, the specific internal energy $e =e(\varrho,\vartheta)$, and the specific entropy  $ s =  s(\varrho,\vartheta)$ are interrelated through Gibbs' relation

\begin{equation}\label{gibbs}
   \vartheta D s (\varrho,\vartheta) = D e(\varrho,\vartheta)+ p(\varrho,\vartheta)D\left(\frac{1}{\varrho}\right).
\end{equation}
If $p,e,s$ satisfy (\ref{gibbs}), in the context of \textit{smooth} solutions to (\ref{Euler}), the second law of thermodynamics is enforced through the entropy balance equation

\begin{equation}\label{entrB}
  \dd (\varrho  s (\varrho,\vartheta))+\mathrm{div}_x( s(\varrho,\vartheta)\vec m)\, \dd t =0, 
\end{equation}
where $s(\varrho,\vartheta)$ denotes the (specific) entropy and is of the form
\begin{equation}\label{entropy}
    s(\varrho,\vartheta)=\log(\vartheta^{c_v})-\log(\varrho).
\end{equation}
Denoting by $$S=S(\varrho,\vartheta)=\varrho s(\varrho,\vartheta)=\varrho\log(\vartheta^{c_v})-\varrho\log(\varrho)$$ the total entropy, equation \eqref{entropy} can also be written as
\begin{equation}\label{EntrB}
  \dd S (\varrho,\vartheta)+\mathrm{div}_x\Big(S(\varrho,\vartheta)\frac{\vec m}{\varrho}\Big)\, \dd t =0.
\end{equation}
For weak solutions, the equalities in (\ref{entrB}) and (\ref{EntrB}) no longer hold, the entropy balance is given as an inequality, for more details see \cite{SBGD}. 
Finally, system \eqref{EulerB} can be reformulated as
\begin{align}
\dd \varrho + \mathrm{div}_x\vec m\,\dd t &=0\qquad\,\,\text{in $Q$},\label{eq:aEuler}\\
\dd  \vec m + \mathrm{div}_x\left(\frac{ \vec m \otimes \vec m}{\varrho}\right)\dd t+\nabla_x p\,\dd t  &= \varrho\phi\dd W\,\,
 \text{in $Q$,}\label{eq:bEuler}\\
\dd S +\mathrm{div}_x\left(\frac{S\vec m}{\varrho} \right)\,\dd t &=0 \quad\qquad\text{in $Q$.}\label{eq:cEuler}
\end{align}
For physical relevant solutions, the problem is augmented by the total energy balance
\begin{equation}\label{eq:tenergy}
      \dd \int_{\T}\mathcal E\, \dd x=  \dd \int_{\T}\left [\frac{1}{2}\varrho|\vec u|^2 +\varrho e\right]\, \dd x= \int_{\T}\varrho\phi \cdot\vec u\,\dd x \dd W+ \frac{1}{2}\int_{\T}\|\phi\|_{\ell_2}^2\dx\,\dd t.
\end{equation}
Strong solutions of the system (\ref{Euler}) satisfy (\ref{eq:tenergy}), but for weak solutions it has to be added in the definition. In fact, weak solutions to not posses sufficient spatial regularity to derive \eqref{eq:tenergy} form the other equations. It needs to be derived on the approximate level and subsequently kept in the limit procedure.
\begin{ex}
A particular instance is given by the two-dimensional standard Wiener process
\begin{align}\label{eq:14.02}
\phi W=\begin{pmatrix} 1\\0\end{pmatrix}\beta_1+\begin{pmatrix} 0\\1\end{pmatrix}\beta_2
\end{align}
with two independent standard one-dimensional Wiener processes $ \beta_1$ and $\beta_2$, where $\|\phi\|_{\ell^2}^2=2$. We will apply this noise in our numerical section \ref{sec_numerical} multiplied with some constant to handle the strength and so the influence of it. 
\end{ex}
\subsection{Probability framework}
\label{sec:prob}
Let $(\Omega, \FF, (\FF_t)_{t\geq 0},\p)$ be a complete stochastic basis with a probability measure $\p$ on $(\Omega,\FF)$ and right-continuous filtration $(\FF_t)_{t\geq 0}$.
The filtration is a family of $\sigma$-sub-algebras of $\FF$ indexed by time $t\geq0$, where each $\FF_t$ represents the information available up to time $t$. In other words, it is the collections of events that are \emph{known}  at time $t$. 
 Let $\UU$ be a separable Hilbert space with orthonormal basis $(e_k)_{k\in \N}$ (a natural choice would be $\UU=L^2(\mathcal O)$). We denote by $ L_2(\UU,L^2(\T))$ the set of Hilbert-Schmidt operators from $\UU$ to $L^2(\T)$, i.e., the set of bounded linear mappings $\Phi:\UU\rightarrow L^2(\mathcal O)$ satisfying
 \begin{equation}\label{eq:HSa}
    \sum_{k\geq 1}\|\phi(e_k)\|_{L^{2}(\T)}^2<\infty.
\end{equation}
 The stochastic process $W$ appearing in \eqref{EulerB} is a cylindrical Wiener process $W =(W_t)_{t \geq 0}$ in $\UU$, meaning it is of the form
\begin{equation}
W(s) = \sum_{k\in \N}e_k \beta_k(s),
\label{Dwiener}
\end{equation}

where $(\beta_k)_{k \in \N}$ is a sequence of independent real-valued one-dimensional Wiener processes relative to $(\FF_t)_{t\geq0}$. To identify the precise definition of the diffusion coefficient, set $\UU =\ell^2$ and consider $\varrho \in L^{1}(\T), \varrho > 0$. For a mapping $\phi \in L_2(\UU,L^2(\T))$, we set
$
\phi_k=\phi(e_k)\in L^2(\T)
$, $k\in\N,$
and suppose that
\begin{equation}\label{eq:HS}
    \sum_{k\geq 1}\|\phi_k\|_{L^{\infty}(\T)}^2= \sum_{k\geq 1}\|\phi(e_k)\|_{L^{\infty}(\T)}^2<\infty.
\end{equation}
Consequently, since $\phi_k$ is bounded we deduce 
\begin{equation}\label{fyB}
\|\sqrt{\varrho}\phi_k\|_{L_2(\UU, L^2(\T))}^2\lesssim c(\phi)( \|\varrho\|_{L^{1}(\T)}).
\end{equation}
The stochastic integral 
\[
\int_{0}^{\tau} \varrho \phi \,\dd W = \sum_{k\geq1}\int_{0}^{\tau}\varrho\phi(e_k)\, \dd \beta_k,
\]
takes values in the Banach space  $C([0,T];W^{-\mathfrak k,2}(\T))$ in the sense that
\begin{equation}
    \int_{\T}\left(\int_{0}^{\tau}\varrho\phi \,\dd W\cdot \varphi \right)\,\dd x =\sum_{k\geq1}\int_{0}^{\tau}\left(\int_{\T}\varrho\phi(e_k)\cdot \varphi \,\dd x\right)\,\dd \beta_k, \quad \varphi \in W^{\mathfrak k,2}(\T), \mathfrak k>\frac{n}{2}.
\end{equation}
Finally, we define the auxiliary space $\UU_0$ with  $\UU \subset \UU_0$ as 
\begin{eqnarray}
\UU_0 :&=&\Bigg\{ e = \sum_{k}\alpha_k e_k:\sum_{k}\frac{\alpha_{k}^{2}}{k^2} < \infty \Bigg\}, \nonumber\\
\| e\|_{\UU_0}^{2}:& =& \sum_{k}^{\infty}\frac{\alpha_{k}^{2}}{k^2}, \ e=\sum_{k}\alpha_ke_k,
\label{auxiliary}
\end{eqnarray}
so that the embedding $\UU \hookrightarrow \UU_0$ is Hilbert-Schmidt, and the trajectories of $W$ belong $\p$-a.s. to the class $C([0,T];\UU_0)$ (see \cite{Prato}).

\subsection{Discontinuous Galerkin schemes}\label{sec:DG}
As mentioned in Subsection \ref{sec:cr}, we focus on the two-dimensional setting.  
The spatial domain $\T \subset \R^2$ is then discretized  with a mesh of tensor-product elements, e.g., a regular quadrilateral grid,  denoted by $\TT_h$. We denote the generic cell $K$ and the uniform mesh size with $h$. They  are given by 
$$\mathcal{K}:= [x_{i-1/2,j}, x_{i+1/2,j}] \times [y_{i,j-1/2}, y_{i,j+1/2}]$$
with
$h:= x_{i+1/2,j}-x_{i-1/2,j}=y_{i,j+1/2}-y_{i,j-1/2}$, for simplicity. 
Extensions to  (unstructured) rectangular meshes with cell sizes $h_x\neq h_y$ are straightforward. 
With $\partial \mathcal{K}$ we denote the boundary of an element $\mathcal{K}$ and by $\EE$ the set of all interfaces of all cells $\mathcal{K} \in \TT_h$ where $\ee$ is one interface of $\partial \mathcal{K}$. Between two elements $\mathcal{K}^-$ and $\mathcal{K}^+$ we have a normal vector $\bn$. It is given either by $\bn=(n_x, 0)$ or by $\bn=(0,n_y)$ depending on the interface. 
Let $\mathcal{Q}^{\mathfrak p}([-1,1]^2)$  be the space of all multivariate 
polynomials of degree at most $\mathfrak p\in\N_0$ in each variable.
On each element $\mathcal{K} \in \TT_h$, we have a linear map  $T_{\mathcal{K}}:[-1,1]^2\to \mathcal{K}$  and    $\mathcal{Q}^{\mathfrak p}(\mathcal{K})$  is  spanned by functions $\psi \circ T_{\mathcal{K}}^{-1}$ with $\psi\in \mathcal{Q}^{\mathfrak p}([-1,1]^2)$.
The DG solution space $\VV^h$ is given by
 \begin{equation}\label{eq:spaceA}
   \VV^h = \left\{ v^h\in L^1(\T)  \Big| v^h|_{\mathcal{K}} \in \mathcal{Q}^{\mathfrak p}(\mathcal{K}) \text{ for all } \mathcal{K}\in \TT_h \right\}.
 \end{equation}
Approximate solutions to the Euler equations \eqref{Euler} live in  the vector version of this 
 space, that is $\tilde{\VV}^h= [\VV^h]^4$. 
 To describe an element of the finite dimensional space $\VV^h$, we apply a nodal Gauss-Lobatto basis. We denote by  $\xi_i$ the Gauss-Lobatto points in the interval $[-1,1]$. Further,  $L_i$ is the Lagrange polynomial which fulfils $L_i(\xi_j)=\delta_{i,j}$. Here,  $\delta_{i,j}$ is the Kronecker delta. Lagrange polynomials form a basis  for $\mathcal{Q}^{\mathfrak p}([-1,1])$ in one dimension. We obtain a basis for
 $\mathcal{Q}^{\mathfrak p}([-1,1]^2)$  via the tensor product of the one-dimensional basis, i.e., $L_i(x)L_j(y)$.
 In each element, each component of the conservative variable vector of the Euler equations is approximated by a polynomial in the reference space \eqref{eq:spaceA}. The nodal values (interpolation points) are our degrees of freedom (DOFs) where we have to calculate the time-dependent nodal coefficients for all components (density, momentum, energy) in the following.
  To make this point clear, our numerical solution  is given by $\bU^h=(\varrho^h, \bm^h,\mathcal E^h)^\top \in \tilde{\VV}^h$ 
  and each component is represented by a polynomial, e.g. in the reference element the approximated density is given by 
  \begin{equation}\label{eq:approx}
  \varrho^h(x,y,t) = \sum_{i,j=0}^{\mathfrak p} \hat{\varrho}^h(\xi_i, \nu_j, t) L_i(x) L_j(y)= \sum_{i,j=0}^{\mathfrak p} \hat{\varrho}_{i,j}(t) L_i(x) L_j(y)
  \end{equation}
where $ \xi$ and $\nu$ are our Gauss-Lobatto nodes in $x-$ and $y-$direction
and 
$\hat{\varrho}_{i,j}(t)$ is the time-dependent coefficient of our polynomial presentation, i.e. $\hat{\varrho}_{i,j}(t)= \varrho^h(\xi_i, \nu_j,t)$. 
Analogous notation holds for the other components of $\bU^h$. \\
Next, we explain our DG discretization for the stochastic Euler system  \eqref{EulerB}. To avoid the non-conservative products inside the flux formulation which can yield to problems and needs special treatment \cite{abgrall2018,abgrall2010, renac_2019}, we use equation \eqref{eq_total}   instead of the internal energy equation in \eqref{EulerB}. Therefore, the third conservative variable is the total energy  $\mathcal E$ and we consider the system

\begin{align}\label{EulerC}
\begin{aligned}
\dd \varrho + \mathrm{div}_{x}(\vec m)\, \dd t &=0 \quad \text{in}\, Q,\\
\dd \vec m+ \mathrm{div}_x\Big(\frac{\vec m \otimes \vec m}{\varrho}\Big)\, \dd t + (\gamma-1)\nabla_x \bigg(\mathcal E-\frac{1}{2}\frac{|\vec m|^2}{\varrho}\bigg)\,\dd t&=\varrho \phi \,\dd W\quad \text{in}\, Q,\\
\dd \mathcal E + \mathrm{div}_x\left(\left(\mathcal E+(\gamma-1)\Big(\mathcal E-\frac{1}{2}\frac{|\vec m|^2}{\varrho}\Big)\right)\frac{\vec m}{\varrho}\right)\dd t &=\frac{1}{2}\varrho\|\phi\|_{\ell_2}^2 \,\dd t+ \phi\cdot \vec m\,\dd W\quad \text{in}\, Q.\end{aligned}
\end{align}
Or in shorter terms
\begin{align}\label{EulerC_short}
\dd\bU+\mathrm{div}(\mathbf{f}(\bU))\dt=\bh(\bU)\dt+\bg(\bU)\,\dd W,
\end{align} 
where
\begin{align*}
\mathbf{f}(\bU)&=\mathbf{f}\begin{pmatrix}\varrho\\\vec m\\\mathcal E
\end{pmatrix}=\begin{pmatrix}\vec m\\\frac{\vec m \otimes \vec m}{\varrho}+ (\gamma-1)\bigg(\mathcal E-\frac{1}{2}\frac{|\vec m|^2}{\varrho}\bigg)\mathbb I_{2\times 2}\\\left(\mathcal E+(\gamma-1)\Big(\mathcal E-\frac{1}{2}\frac{|\vec m|^2}{\varrho}\Big)\right)\frac{\vec m}{\varrho}\end{pmatrix},\\
\mathbf{h}(\bU)&=\mathbf{h}\begin{pmatrix}\varrho\\\vec m\\\mathcal E
\end{pmatrix}=\varrho\begin{pmatrix}0\\0\\\tfrac{1}{2}\|\phi\|_{\ell_2}^2 \end{pmatrix},\quad\mathbf{g}(\bU)=\mathbf{g}\begin{pmatrix}\varrho\\\vec m\\\mathcal E
\end{pmatrix}=\begin{pmatrix}0\\\varrho \phi \\\phi\cdot \vec m\end{pmatrix}.
\end{align*}
In the case of the two-dimensional noise from \eqref{eq:14.02} we have
\begin{align}\label{eq:14.02b}
\mathbf{h}(\bU)=&\varrho\begin{pmatrix}0\\0\\ 0 \\1\end{pmatrix},\quad\mathbf{g}(\bU)\dd W=\begin{pmatrix}0\\\varrho\, \dd\beta_1\\\varrho\,\dd \beta_2 \\m_1\,\dd\beta_1+m_2\,\dd\beta_2\end{pmatrix}.
\end{align}
If the noise is truly infinite dimensional, it has to be approximated for practical purposes as outlined in Remark \ref{rmk:noise} below.

 We  start by defining the DG discretization for the flux term $\mathrm{div}(\mathbf{f}(\bU))$ in the Euler equation \eqref{EulerC}. 
Let $(\Omega, \FF, (\FF_t)_{t\geq 0},\p)$ be a complete stochastic basis with a probability measure $\p$ on $(\Omega,\FF)$ and right-continuous filtration $(\FF_t)_{t\geq 0}$, $W$ an $(\mathfrak F_t)$-adapted (cylindrical) Wiener process and $\phi\in L_2(\mathfrak U;L^2(\T))$.

 Multiplying by a  test function $\bV^h \in \tilde{\VV}^h$ and integrate over the domain $\T$ it reads as
 \begin{align}\label{eq:DG_strong_analytical}
\begin{aligned}
  \dd\int_{\T} \bU^h \cdot \bV^h \diff \bx& +\sum_{K\in \TT_h} \int_K
  (\mathrm{div}\,\mathbf{f} (\bU^h) )\cdot \bV^h \diff \bx\dt \\=&- \sum_{\partial K^- \in \TT_h} \int_{\partial K^-} (\bfnum(\bU^{h,-}, \bU^{h,+},\bn^-) - \mathbf{f}(\bU^h) \cdot \bn^-) \cdot  \bV^{h,-}\,\dd\mathcal H^1\dt\\
&+\int_{\T}\bh(\bU^h) \cdot \bV^h \diff \bx\,\dd t+\int_{\T}\bg(\bU^h) \cdot \bV^h \diff \bx\,\dd W,
\end{aligned}
 \end{align}
where $\mathbf{f}(\bU^h)=(\mathbf{f}_1,\mathbf{f}_2)$ and for the numerical flux $\bfnum$ we choose
the local Lax-Friedrichs flux given by 
\begin{equation}\label{eq:LLF}
 f^{\mathrm num}(\bU^{h,-}, \bU^{h,+}, \bn^-):= \frac{1}{2} \left(\mathbf{f}(\bU^{h,-}) + \mathbf{f}(\bU^{h,+}) \right)\cdot \bn^- - \frac{\lambda}{2} \left( \mathbf{f}(\bU^{h,+})-\mathbf{f}(\bU^{h,-}) \right),
\end{equation}
where $\lambda \geq \lambda_{max}$ is an upper bound for the maximal wave speed. For the Euler equations \eqref{EulerC}, it holds that 
$\lambda_{max}=\max\{ |\bu^h_{K^-}|+c_{K^-},|\bu^h_{K^+}|+c_{K^+} \}$ with $c=\sqrt{\gamma \frac{p^h}{\varrho^h} }$.\\
Note that  an inner product denoted by $\cdot$ means that \eqref{eq:DG_strong_analytical}
 has to be solved for each component of the Euler equation separately. 
Later all the calculations in our high-order DG method are performed in a reference element $I=[-1,1]^2$.
To evaluate the integrals, we proceed by choosing the Gauss-Lobatto nodes as collocated quadrature points. 
Finally, by this selection of a tensor-product basis and Gauss-Lobatto quadrature, the DG operators have a Kronecker-product structure as well. We denote by $\mat{M}_1$ on $[-1,1]$ the one-dimensional mass matrix.
It has diagonal form with quadrature weight on the diagonal.  Further, we obtain the one-dimensional differentiation matrix $\mat{D}_{1}$  by evaluating the derivatives of the basis functions at the nodal points.  We use the index here to clarify that we have the one-dimensional setting and working on the reference element $I$. 
To clarify the setting, let $\xi_j$ be the Gauss-Lobatto quadrature points $-1=\xi_0<\xi_1<\cdots<\xi_{\mathfrak p}=1$ in $[-1,1]$
with corresponding quadrature weights 
 $\{\omega_j\}_{j=0}^{\mathfrak p}$. The nodal Lagrangian basis is given by  $L_j(\xi_l)=\delta_{jl}$ and we can define 
the discrete inner $\langle u,v\rangle_{\omega}:=\sum_{j=0}^{\mathfrak p} \omega_j u(\xi_j) v(\xi_j)$ in one space-dimension.  Then, the above described operators are given by 
 \begin{itemize}
\item Difference matrix $\mat{D}_1$ with $\mat{D}_{1,jl} =L_l'(\xi_j)$ with $L_l'$ be the first derivative of the $l-$Lagrange polynomial,
\item Mass matrix $\mat{M}_{1,jl}=\langle L_j,L_l\rangle_{\omega} =\omega_j \delta_{jl}$, so that
   $\mat{M}_1 = \diag \{\omega_0, \dots, \omega_{\mathfrak p}\}$.
\end{itemize}
Later, we need also the operators 
\begin{itemize}
\item Stiffness matrix $\mat{Q}_{1,jl}=\langle L_j',L_l\rangle_{\omega} =\langle L_j,L_l'\rangle_{\omega}$
\item Interface matrix $\mat{B}_1=\diag(-1,0,\cdots,0,1)$
\end{itemize}
Up to this point, these operators would work separately on every component of the conservative variable vector $\bU^h$, i.e. 
on the density $\varrho^h$, momentum  $\bm^h$, and energy $\mathcal{E}^h$. To extend these operators to the Euler system, a simple Kronecker product can be used. It is
 \begin{equation*}
   \mat{\mathbf{M}}_1=\mat{M}_1 \otimes \IdxM \qquad \mat{\mathbf{D}}_1=\mat{D}_1 \otimes \IdxM,
 \end{equation*}
where $ \IdxM$ is the $4\times4$-identy matrix.\\
Similar in two-dimension, we obtain the local mass matrix and differentiation matrix in the standard element $I$ through Kronecker products:
  $\mat{\mathbf{M}}_I= \mat{\mathbf{M}}_1\otimes  \mat{\mathbf{M}}_1,
  \mat{\mathbf{D}}_{1,I}= \IdxM\otimes \mat{\mathbf{D}}_1,
  \mat{\mathbf{D}}_{2,I}= \mat{\mathbf{D}}_1 \otimes \IdxM.$
To solve \eqref{eq:DG_strong_analytical}, we have to  evaluate the 
cell interface integrals. Due to the tensor structure ansatz, we evaluate at each interface the one-dimensional Gauss-Lobatto quadrature rule. Due to the Kronecker ansatz, we obtain the
interface operators 
$
\mat{\mathbf{B}}_{\ee_1}=  \mat{\mathbf{M}}_1 \otimes \mat{\mathbf{B}}_1 \text{ and }  \mat{\mathbf{B}}_{\ee_2}= \mat{\mathbf{B}}_1 \otimes  \mat{\mathbf{M}}_1 
$
depending on the considered cell interfaces. The defined operators fulfil the summation-by-parts property meaning that they mimic discretely integration-by-parts. \\
In  \eqref{eq:DG_strong_analytical}, we have to calculate  the time-dependent coefficients of our polynomial representation \\
$\bU^h= (\varrho^h, \bm^h, \mathcal{E}^h) \in\tilde{\VV}^h.$ We have to distinguish in numbering between the interpolation (quadrature) points in $x$- and $y$- direction in \eqref{eq:approx}. For simplicity, we are renumbering the points. We have $n_{\mathfrak p}=({\mathfrak p}+1)^2$ quadrature/interpolation points\footnote{Alternatively, a multi-index can be used for the notation.} denoted by $\boldsymbol{\xi}_i$ in each element and $n_{b}=4{\mathfrak p}$ at the interfaces. On each face we have ${\mathfrak p}+1$ Gauss-Lobatto quadrature nodes but on the corners they intersect. The basis is given by $\mathbf{L}_i$ with $i \in \{1,n_{\mathfrak p}\}$.
 We denote by $\bvu$ the vector of coefficients (i.e., the nodal values) of $\bU^h$ on $K$:
\begin{equation}\label{eq:coefficients}
\bvu= \left(\varrho^h(\boldsymbol{\xi}_1), \bm^h(\boldsymbol{\xi}_1), \mathcal{E}^h(\boldsymbol{\xi}_1 );
\varrho^h(\boldsymbol{\xi}_2), \bm^h(\boldsymbol{\xi}_2),  \mathcal{E}^h(\boldsymbol{\xi}_2 ); \dots ;\varrho^h(\boldsymbol{\xi}_{n_{\mathfrak p}}), \bm^h(\boldsymbol{\xi}_{n_{\mathfrak p}}), \mathcal{E}^h(\boldsymbol{\xi}_{n_{\mathfrak p}}) \right)^T.
\end{equation}
Let $ \underline{\mathbf{f}}_m$ ($m=1,2$) denote the vector of values of $\mathbf{f}_m(\bU^h)$ evaluated at the nodal points. For each cell interface $\ee \in \partial K \subset \EE$, we have to evaluate $
\mathbf{f}(\bU^h)$ at the one-dimensional Gauss-Lobatto nodes on the cell interface $\ee$ (where the trace of $\bU^h$ is taken from inside $K$ dotted with the scaled normal vector $\bn$ facing outwards from $\ee$). We denote this by $\mat{\mathbf{R}}_{\ee_m}\underline{\mathbf{f}}_{\ee_m}$. Likewise $ \underline{\mathbf{f}}^{\mathrm{num}}_{\ee_m}$ denotes the nodal values of $\bfnum(\bU^{h,-}, \bU^{h,+},\bn^-)$.
With these operators, we can finally re-write the DG semidiscretization  \eqref{eq:DG_strong_analytical} on the reference element as follows
\begin{align}\label{eq:DG_standard}
\begin{aligned}
 \dd \bvu &+  \big(\mat{\mathbf{D}}_{1,I}  \underline{\mathbf{f}}_1 +\mat{\mathbf{D}}_{2,I}  \underline{\mathbf{f}}_2\big)\dt\\&\qquad\qquad= \mat{\mathbf{M}}_I^{-1}  \sum_{j \in \partial I} \mat{\mathbf{B}}_{j} \left(\mat{\mathbf{R}}_{j}\underline{\mathbf{f}}_{j}-\underline{\mathbf{f}}^{\mathrm{num}} \right)\dt+\underline\bh\,\dd t+\underline\bg\,\dd W,
\end{aligned}
\end{align}
where $\partial I$ denotes the cell interfaces of the reference element. 
Equation \eqref{eq:DG_standard} describes the classical discontinuous Galerkin spectral element method (DGSEM) in two-space dimension.
The method is by construction not  entropy conservative/ dissipative for the deterministic Euler equations \cite{zbMATH07517718,gassner2016}. To obtain an high-order entropy dissipative DG method for the Euler equation, we apply the flux differencing approach \cite{fischer_2013,gassner2016}. To this end we replace the volume flux  $\sum_{m=1}^2 \mat{\mathbf{D}}_{m,I}  \bvf_m $  in equation \eqref{eq:DG_standard} above using consistent, symmetric two-point  numerical fluxes. The resulting DG scheme in the reference element reads than 
 \begin{align}\label{eq:DG_standard_2}
\begin{aligned}
 \dd \bvu &+ 2 \left( \mat{\mathbf{D}}_{1,I} \underline{\mathbf{f}}^{\mathrm{num}}_{1,Vol}(\bvu, \bvu)+\mat{\mathbf{D}}_{2,I} \underline{\mathbf{f}}^{\mathrm{num}}_{2,Vol}(\bvu, \bvu) \right)\dt\\&\qquad\qquad= \mat{\mathbf{M}}_I^{-1}  \sum_{j \in \partial I} \mat{\mathbf{B}}_{j} \left(\mat{\mathbf{R}}_{j}\bvf_{j}-\underline{\mathbf{f}}^{\mathrm{num}}_j \right)\dt+\underline\bh\,\dd t+\underline\bg\,\dd W,
\end{aligned}
\end{align}
where $\underline{\mathbf{f}}^{\mathrm{num}}_{Vol}$ denotes the numerical volume flux working on each degree of freedom and $\underline{\mathbf{f}}^{\mathrm{num}}_j $ is the classical numerical flux at the interface. 
For the numerical volume flux $\mathbf{f}^{\operatorname{num}}_{m,Vol}$ with $m=1,2$ , we select the consistent, symmetric and entropy conservative two-point flux from \cite{ranocha2018}. It is defined for each component separately: 
 \begin{equation}\label{eq:Ranocha_flux} 
 \begin{aligned}
  f^{\mathrm num}_{\varrho,1} &=\{\{\varrho\}\}_{\log}\armean{u_1}, \quad  f^{\mathrm num}_{\varrho u_1,1}  =\armean{ u_1} f^{\mathrm num}_{\varrho,1} +\armean{ p}, \quad  f^{\mathrm num}_{\varrho u_2,1}=\armean{ u_2} f^{\mathrm num}_{\varrho} \\
  f^{\mathrm num, x}_{E,1}&= \left( \{\{\varrho\}\}_{\log}  \left( \armean{ u_1}^2+\armean{ u_2}^2 - \frac{\armean{ u_1+u_2}^2}{2} \right) -\frac{1}{\gamma-1} \frac{\logmean{ \varrho} }{\logmean{ \varrho/p} } +
  \armean{ p} \right) \armean{u_1} \\
  &-\frac{\jump{p} \jump{\bu}}{4},
 \end{aligned}
 \end{equation}
 with $\mathbf{f}^{\operatorname{num}}_2$ defined analogously. Here, we have used the abbreviations $\armean{\varrho}=\frac{\varrho^{+}+\varrho^-}{2} $ and
 $\logmean{\varrho}=\frac{\varrho^{+}-\varrho^-}{\log \varrho^+-\log \varrho^-}$.

\begin{rmk}[Interpretation and solvability]
\begin{itemize}
\item Discretization \eqref{eq:DG_standard_2} is classical entropy stable DGSEM method  for the Euler equation equipped with two additional forcing terms due to the noise. The first term $\bh$ has no stochastic component and can be interpreted as a simple forcing of the total energy  whereas the second part $\bg$ acts on the momentum equation and the total energy and takes the stochasticity into account.
\item The well-posedness assumption regarding \eqref{eq:DG_standard_2} can be substituted with the condition that the coefficients of the nonlinear SDE system are Lipschitz continuous, as noted in \cite{KuLu} for the deterministic parts. This substitution is directly deducible from the definition of the Euler fluxes and the assumptions concerning the positive values of pressure and density. The noise coefficients are linear and thus clearly Lipschitz as well. Hence standard results yield the existence of a solution defined on a given complete stochastic basis $(\Omega, \FF, (\FF_t)_{t\geq 0},\p)$ with respect to the $(\FF_t)$-Wiener process $W$ and given initial data (which we assume for simplicity to be deterministic).
\end{itemize}
\end{rmk}

\section{Concept of solutions and main results}\label{se_concept}

\subsection{Measure-valued solutions}\label{MVS}

We are now ready to introduce the concept of measure-valued martingale solutions to the \textit{complete} stochastic Euler system written in entropy conservative variables (\ref{eq:aEuler})--(\ref{eq:cEuler}). From here onward, we denote by $\mathcal{M}^+$ the space of non-negative Radon measures, and we denote by  $A$ the space of ``dummy variables":

\begin{equation}\label{eq:space}
A = \bigg\{[\varrho', \vec m',{S'}]\bigg|\varrho'\geq 0, \vec m' \in \R^3,S'\in \R\bigg\}
\end{equation}
We denote by $\mathcal{P}(A)$ the space of probability measures on $A$. We follow \cite{Mo} and define: 

\begin{dfn}[Dissipative measure-valued martingale solution]\label{E:dfn}Let $\varrho_0\in L^{\gamma}(\T)$, $\vec m\in L^{\frac{2\gamma}{\gamma +1}}(\T)$, $S_0\in L^{\gamma}(\T)$ and $\phi \in L_2(\UU;L^2(\T))$. Then 
\[
((\Omega,\FF, (\FF_t)_{t\geq 0},\p),\varrho,\vec m,S, \mathcal{R}_{\text{conv}},\mathcal{R}_{\text{press}},\mathcal{V}_{t,x},\mathfrak t, W)
\]
is called a dissipative measure-valued
martingale solution to (\ref{eq:aEuler})--(\ref{eq:cEuler}) with initial law $\Lambda$, provided\footnote{Some of our variables are not stochastic processes in the classical sense and we interpret their adaptedness in the sense of random distributions as introduced in \cite{FrBrHo} (Chap. 2.2).}:

\begin{enumerate}
    \item [(a)] $(\Omega,\FF, (\FF_t)_{t\geq 0},\p)$ is a stochastic basis with complete right-continuous filtration;
\item[(b)] $\mathfrak t$ is an $(\FF_t)_{t\geq 0}$-stopping time with $\p(\mathfrak t>0)=1$;
    \item[(c)]$W$ is a $(\FF_t)_{t\geq 0}$-cylindrical Wiener process;
    \item[(d)] The density $\varrho$ is $(\FF_t)_{t\geq 0}$-adapted and satisfies $\p$-a.s.
    \[
    \varrho(\cdot\wedge\mathfrak t) \in C_{\text{loc}}([0,\infty), W^{-4,2}(\T))\cap L_{\text{loc}}^{\infty}(0,\infty;L^{\gamma}(\T));
    \]
    \item[(e)] The momentum $\vec m$ is $(\FF_t)_{t\geq 0}$-adapted and satisfies $\p$-a.s.
    \[
    \vec m(\cdot\wedge\mathfrak t) \in C_{\text{loc}}([0,\infty), W^{-4,2}(\T))\cap L_{\text{loc}}^{\infty}(0,\infty;L^{\frac{2\gamma}{\gamma +1}}(\T));
    \]
    \item[(f)] The total entropy $S$ is $(\FF_t)_{t\geq 0}$-adapted and satisfies $\p$-a.s.
    \[
    S(\cdot\wedge\mathfrak t) \in L^\infty([0,\infty),L^{\gamma}(\T))\cap BV_{w,\text{loc}}(0,\infty;W^{-l,2}(\T)), l>\frac{n+2}{2};
    \]
    \item[(g)] The parametrised measures $(\mathcal{R}_{\text{conv}},\mathcal{R}_{\text{press}},\mathcal{V})$ are $(\FF_t)_{t\geq 0}$-progressively measurable and satisfy $\p$-a.s.
    \begin{eqnarray}\label{eq:para}
    t\mapsto \mathcal{R}_{\mathrm{conv}}(t\wedge\mathfrak t) &\in &L_{\text{weak-(*)}}^{\infty}(0,\infty;\mathcal{M}^+(\T, \mathbb{R}^{3\times 3}));\\
    t\mapsto \mathcal{R}_{\mathrm{press}}(t\wedge\mathfrak t) &\in &L_{\text{weak-(*)}}^{\infty}(0,\infty;\mathcal{M}^+(\T, \mathbb{R}));\\
    (t,x)\mapsto \mathcal{V}_{t\wedge\mathfrak t,x} &\in &L_{\text{weak-(*)}}^{\infty}(Q;\mathcal{P}(A));
    \end{eqnarray}
    \item[(h)] We have $\varrho(0,\cdot)=\varrho_0$, $\vec m(0,\cdot)=\vec m_0$ and $\liminf_{\tau\rightarrow0+}S(\tau,\cdot)\geq S(\cdot,0)$ $\mathbb P$-a.s.;
    \item[(i)]For all $\varphi \in C^{\infty}(\T)$ and all $\tau > 0$ there holds $\p$-a.s.
    \begin{equation}\label{eq:cont}
       \left[\int_{\T}\varrho\varphi\, \dd x\right]_{t=0}^{t=\tau\wedge\mathfrak t} = \int_{0}^{\tau\wedge\mathfrak t}\int_{\T} \vec m \cdot \nabla \varphi \, \dd x\dd t;
    \end{equation}
    \item[(j)]For all $\boldsymbol{\varphi} \in C^{\infty}(\T)$ with $\boldsymbol{\varphi}\cdot\mathbf n=0$ and all $\tau > 0$ there holds $\p$-a.s.
    \begin{eqnarray}\label{eq:mcxs}
         \left[\int_{\T}\vec m \cdot \boldsymbol{\varphi}\right]_{t=0}^{t=\tau\wedge\mathfrak t}&=&\int_{0}^{\tau\wedge\mathfrak t}\int_{\T}\left[\frac{\vec m \otimes\vec m}{\varrho}:\nabla\boldsymbol{\varphi}+\varrho\exp{\left(\frac{S}{c_v\varrho}\mathrm{div}\boldsymbol{\varphi}\right)}\right]\, \dd x\, \dd t\\
         &&+\int_{0}^{\tau\wedge\mathfrak t}\nabla \varphi:\dd \mathcal{R}_{\text{conv}} \dd t+\int_{0}^{\tau\wedge\mathfrak t}\int_{\T}\mathrm{div} \boldsymbol{\varphi}\,\dd \mathcal{R}_{\text{press}} \dd t\nonumber\\
         &&+\int_{0}^{\tau\wedge\mathfrak t}\boldsymbol{\varphi}\cdot \varrho \phi \, \dd x\, \dd W;\nonumber
    \end{eqnarray}
    \item[(k)] The total entropy holds in the sense that 
    \begin{equation}\label{eq:entr}
        \int_{0}^{\tau\wedge\mathfrak t}\int_{\T}\left[\langle\mathcal{V}_{t,x};S'\rangle\partial_t \varphi+\langle\mathcal{V}_{t,x},S'\frac{\vec m'}{\varrho'}\rangle\cdot \varphi\right]\, \dd x\dd t \leq \left[\int_{\T}\langle \mathcal{V}_{t,x};S'\rangle \varphi\, \dd x\right]_{t=0}^{t=\tau\wedge\mathfrak t},
    \end{equation}
    for any $\varphi \in C^1([0,\infty)\times \T), \varphi \geq 0, \p$-a.s.
    \item[(l)] The total energy satisfies
    \begin{equation}\label{eq:Ener}
        \mathcal{E}_{t\wedge\mathfrak t}=\mathcal{E}_{s\wedge\mathfrak t}+\frac{1}{2}\int_{s\wedge\mathfrak t}^{t\wedge\mathfrak t}\|\sqrt{\varrho}\phi\|_{L_2(\UU;L^2(\T))}^2\, \dd \sigma + \int_{s\wedge\mathfrak t}^{t\wedge\mathfrak t}\int_{\T}\vec m \cdot \phi \, \dd x \dd W,
    \end{equation}
    $\p$-a.s. for a.a. $0\leq s<t$, where
    \[
    \mathcal{E}= \int_{\T}\left[\frac{1}{2}\frac{|\vec m |^2}{\varrho}+c_v\varrho^{\gamma}\exp{\left(\frac{S}{c_v\varrho}\right)}\right]\, \dd x + \frac{1}{2}\int_{\T}\dd \,\mathrm{tr}\mathcal{R}_{\text{conv}}(t) +c_v\int_{\T}\dd \,\mathrm{tr}\mathcal{R}_{\text{press}}(t)
    \]
    for $\tau \geq 0$ and 
    \[
   \mathcal E_0= \int_{\T}\left[\frac{1}{2}\frac{|\vec m_0 |^2}{\varrho_0}+c_v\varrho_0^{\gamma}\exp{\left(\frac{S_0}{c_v\varrho_0}\right)}\right]\, \dd x.
    \]
\end{enumerate}
\end{dfn}
The following result is proved in \cite{Mo} (with $\mathfrak t=\infty$) and ensures that dissipative measure-valued martingale solution exists globally.
\begin{thm}\label{ExMainr}
Assume (\ref{eq:HS}) holds. Let $\Lambda$ be a Borel probability measure on $L^{\gamma}(\T)\times L^{\gamma}(\T)\times L^{\frac{2\gamma}{\gamma +1}}(\T)$ such that
\[
\Lambda\bigg\{(\varrho,S,\vec m)\in L^{\gamma}(\T)\times L^{\gamma}(\T)\times L^{\frac{2\gamma}{\gamma +1}}(\T): 0<\underline{\varrho}<\varrho<\overline{\varrho},0<\underline{\vartheta}<\vartheta<\overline{\vartheta} \bigg\}=1,
\]
where $\underline{\vartheta}, \overline{\vartheta},\underline{\varrho}, \overline{\varrho} $ are deterministic constants. Moreover, the moment estimate

\[
\int_{L^{\gamma}(\T)\times L^{\gamma}(\T)\times L^{\frac{2\gamma}{\gamma +1}}(\T)}\left\|\frac{1}{2}\frac{|\vec m|^2}{\varrho}+c_v\varrho^{\gamma}\exp{\left(\frac{S}{c_v\varrho}\right)}\right\|_{L^1(\T)}^p\,\dd \Lambda < \infty,
\]
holds for all $p\geq 1$. Then there exists a dissipative measure-valued martingale solution to the \textit{complete} stochastic Euler system  (\ref{eq:aEuler})--(\ref{eq:cEuler}) in the sense of Definition \ref{E:dfn} (with $\mathfrak t=\infty$) subject to the initial law $\Lambda$. 

It satisfies the entropy balance also in the renormalised sense, that is we have
    \begin{equation}\label{eq:entrren}
        \int_{0}^{\tau\wedge\mathfrak t}\int_{\T}\left[\langle\mathcal{V}_{t,x};Z(S')\rangle\partial_t \varphi+\langle\mathcal{V}_{t,x},Z(S')\frac{\vec m'}{\varrho'}\rangle\cdot \varphi\right]\, \dd x\dd t \leq \left[\int_{\T}\langle \mathcal{V}_{t,x};Z(S')\rangle \varphi\, \dd x\right]_{t=0}^{t=\tau\wedge\mathfrak t},
    \end{equation}
    for any $\varphi \in C^1([0,\infty)\times \T), \varphi \geq 0, \p$-a.s.,{and any  $Z \in BC(\R)$ non-decreasing.} 
\end{thm}

\subsection{Strong solutions}
In the following, we will give the definition of a local strong pathwise solution. These solutions are strong in the probabilistic and PDE sense, at least locally in time. In particular, the Euler system (\ref{Euler}) will
be satisfied pointwise in space-time on the given stochastic basis associated
to the cylindrical Wiener process $W$.

\begin{dfn}[Strong Solution]\label{def:compstrong}
Let $(\Omega, \FF, (\FF_t)_{t\geq 0}, \p)$ be a complete stochastic basis with a right-continuous filtration, let $W$ be an $(\FF_t)_{t\geq 0}$-cylindrical Wiener process, $\phi\in L_2(\mathfrak U;L^2(\T))$ and $\ell\geq4$. The triplet $[r, \Theta, \vec v]$ and a stopping time $\mathfrak{s}$ is called a (local) strong pathwise solution to the system (\ref{Euler}) provided:

\begin{itemize}
    \item[(a)] the density $r> 0$ $\p$-a.s., $t \mapsto r (t\wedge \mathfrak{s},\cdot)\in W^{\ell,2}(\T)$ is $(\FF_t)_{t\geq 0}$-adapted,
    \[
    \E\left[\sup_{t\in [0,T]}\|r (t\wedge \mathfrak{s},\cdot)\|_{W^{\ell,2}(\T)}^q\right]<\infty \quad\text{for all $1\leq q <\infty$, $T>0$},;
    \]
    \item[(b)] the temperature $\Theta > 0$ $\p$-a.s., $t \mapsto \Theta (t\wedge \mathfrak{s},\cdot)\in W^{\ell,2}(\T)$ is $(\FF_t)_{t\geq 0}$-adapted,
    \[
    \E\left[\sup_{t\in [0,T]}\|\Theta (t\wedge \mathfrak{s},\cdot)\|_{W^{\ell,2}(\T)}^q\right]<\infty \quad\text{for all $1\leq q <\infty$, $T>0$};
    \]
    \item[(c)] the velocity $t\mapsto \vec v(t\wedge \mathfrak{s},\cdot) \in W^{\ell,2}(\T)$is $(\FF_t)_{t\geq 0}$-adapted,
    \[
    \E\left[\sup_{t\in [0,T]}\|\vec v (t\wedge \mathfrak{s},\cdot)\|_{W^{\ell,2}(\T)}^q\right]<\infty \quad\text{for all $1\leq q <\infty$, $T>0$};
    \]
    \item[(d)]  for all $t\geq0$ there holds $\p$-a.s.
    \[
    r(t\wedge\mathfrak{s}) = \varrho(0)- \int_{0}^{t\wedge \mathfrak{s}}\mathrm{div}_x(r\vec v)\,\dd t,
    \]
    \[
    (r\vec v)(t\wedge \mathfrak{s})=(r\vec v)(0)-\int_{0}^{t\wedge}\mathrm{div}(r\vec v\otimes \vec v)\, \dd t-\int_{0}^{t\wedge\mathfrak{s}}\nabla_x p(r,\Theta)\,\dd t +\int_{0}^{t\wedge \mathfrak{s}}r\phi\, \dd W,
    \]
    \[
    (rs(r,\Theta))(t\wedge\mathfrak{s}) = (rs(r,\Theta))(0)- \int_{0}^{t\wedge \mathfrak{s}}\mathrm{div}_x(rs(r,\Theta)\vec v)\,\dd t,
    \]
    where $s$ is the total entropy given by (\ref{entropy}).
\end{itemize}
\end{dfn}
Since density and temperature are strictly positive one can rewrite the system from Definition \ref{def:compstrong} (d) equivalently in terms of $\chi:=\log(\varrho)$, the velocity field $\vec v$ and the temperature $\Theta$ obtaining
\begin{align}\label{EulerD}
\begin{aligned}
\dd \Theta + \vec v\cdot\nabla\Theta\dd t +\Theta\,\mathrm{div}(\vec v)\dd t&=0 \quad \text{in}\, Q\\
\dd \chi + \vec v\cdot\nabla\chi+\mathrm{div}_{x}(\vec v)\, \dd t &=0 \quad \text{in}\, Q,\\
\dd \vec v+ (\vec v\cdot\nabla)\vec v\, \dd t + \Theta\nabla_x \chi\,\dd t+\,\nabla\Theta\,\dt&=\phi \,\dd W\quad \text{in}\, Q
.\end{aligned}
\end{align}
This is a symmetric hyperbolic system with additive noise. Given some initial data $(\varrho_0,\bfU_0,\Theta_0)$ we expect the existence of a unique strong pathwise solution provided (setting $\chi_0:=\log(\varrho_0)$)
\begin{align*}
\chi_0,\bfU_0,\Theta_0\in W^{\ell,2}(\mathcal O),\quad\phi\in L_2(\mathfrak U;W^{\ell,2}(\mathcal O)),
\end{align*}
together with no flux boundary conditions for the initial data.
A corresponding result for the problem on the whole space follows from \cite{Kim}.

The following weak (measure-valued)-strong uniqueness principle is proved in \cite{Mo}.
\begin{thm}\label{thm_a} The pathwise weak-strong uniqueness holds true for the system (\ref{eq:aEuler})--(\ref{eq:cEuler}) in the following sense.
 Let 
\[((\Omega , \FF , (\FF)_{t \geq 0}, \mathbb{P} ),\varrho,\vec m,  S, \mathcal{R}_{\mathrm{conv}},\mathcal{R}_{\mathrm{press}},\mathcal{V}_{t,x},W)
\]
 be a dissipative measure-valued martingale solution to (\ref{eq:aEuler})-(\ref{eq:cEuler}) in the sense of Definition (\ref{E:dfn}) (with $\mathfrak t=\infty$) which satisfies additionally the relative entropy balance in the renormalised sense, cf.~\eqref{eq:entrren}. Let the triplet $[r,\Theta,\vec v]$  and a stopping time $\mathfrak{s}$ be a strong solution in the sense of Definition \ref{def:compstrong} of the same problem; defined on the stochastic basis with the same Wiener process and with initial data
 \begin{equation}\label{Idata}
 \varrho(0,\cdot) =r(0,\cdot),\quad\vec u (0,\cdot)={\vec v}(0,\cdot),\qquad \vartheta (0,\cdot)=\vec \Theta(0,\cdot)\quad\p\text{-a.s.}
 \end{equation}
Then
 \[ [{\varrho},{\vartheta},{\vec u}](\cdot \wedge \mathfrak{s}) \equiv [r,{\Theta},{\vec v}] (\cdot \wedge \mathfrak{s}),
\]
and
\[ \mathcal{R}_{\mathrm{conv}}=\mathcal{R}_{\mathrm{press}}=0,   \]
 $\p$-a.s., and for any $(t,x)\in (0,T)\times\T$
 \[
  \mathcal{V}_{t\wedge \mathfrak{s},x}= \delta_{ s(r,\Theta)},
\]
$\p$-a.s.
\end{thm}
\begin{rmk}
If $\mathfrak p=0$ one can
argue as in \cite{feireisl2021numerical}
to show that the solution obtained by our algorithm (cf. Theorem \ref{thm:main1} below) satisfies the relative entropy balance in the renormalised sense, cf.~\eqref{eq:entrren}, such that Theorem \ref{thm_a} applies. For $\mathfrak p\geq1$ this is currently unclear. Note that this point does not seem to be affected by the noise, this issue appears already in the deterministic case.
\end{rmk}

\subsection{Consistency and convergence}
In this subsection we formulate our main results concerning the DG scheme from \eqref{eq:DG_standard_2} with solution $(\varrho^h,\vec m^h,\mathcal E^h)$ posed on a stochastic basis $(\Omega, \FF, (\FF_t)_{t\geq 0},\p)$ regarding its approximation properties for the Euler equations \eqref{Euler}. We work under the following hypothesis: There is an $(\mathfrak F_t)$-stopping time $\mathfrak t$ with $\p(\mathfrak t>0)=1$ and deterministic constants $K>0$ such that $\p$-a.s.
\begin{align}\label{ass:main}
\inf_{t\in[0,\mathfrak t],x\in\T}\varrho^h(t,x)\geq \frac{1}{K},\quad \sup_{t\in[0,\mathfrak t],x\in\T}\mathcal E^h(t,x)\leq K,
\end{align}
uniformly in $h$. This assumption is the natural counterpart of the corresponding hypothesis in the deterministic case made in various papers such as  \cite{AbLu,feireisl2021numerical, KuLu, LuYu}.

Since the scheme \eqref{eq:DG_standard_2} is formulated in terms of the total energy $\mathcal E^h$, while Definition \ref{E:dfn} is formulated in terms of the total entropy $S$ we need to introduce the approximate total entropy $S^h$ given by
\begin{align*}
S^h=c_v\varrho^h\log\bigg((\gamma-1)\bigg(\mathcal E^h-\frac{1}{2}\frac{|\vec m^h|^2}{\varrho^h}\bigg)\bigg)-(c_v+1)\varrho^h\log(\varrho^h).
\end{align*}
Also, the approximate temperature $\vartheta^h$ is given as
\begin{align*}
\vartheta^h=\frac{\gamma-1}{\varrho^h}\bigg(\mathcal E^h-\frac{1}{2}\frac{|\vec m^h|^2}{\varrho^h}\bigg).
\end{align*}
Now we formulate our main result concerning the consistency of the scheme \eqref{eq:DG_standard_2}.
\begin{thm}\label{thm:main1}
Let $(\Omega, \FF, (\FF_t)_{t\geq 0},\p)$ be a complete stochastic basis with a probability measure $\p$ on $(\Omega,\FF)$ and right-continuous filtration $(\FF_t)_{t\geq 0}$ and $W$ an $(\FF_t)$-adapted Wiener process. Suppose that $\phi\in L_2(\mathfrak U;L^2(\T))$ satisfies \eqref{eq:HS}. Let  $(\varrho^h,\vec m^h,\mathcal E^h)$ be the solution to \eqref{eq:DG_standard_2}. Suppose there is an $(\mathfrak F_t)$-stopping time $\mathfrak t$ such that \eqref{ass:main} holds. Suppose that the initial data $(\varrho_0,\vec m_0,S_0)$ satisfies
\begin{align*}
\varrho_0,S_0\in\ L^\gamma(\T),\quad\vec m_0\in L^{\frac{2\gamma}{\gamma+1}}(\T),\quad\inf_{x\in\T}\varrho_0>K,\quad \sup_{x\in\T}\mathcal E_0<K.
\end{align*}
 Then it there is a null-sequence $(h_m)$ such that for all $q<\infty$
    \begin{equation}
        \begin{cases}
        {\varrho}^{h_m} \to \Tilde{\varrho} \,\, \text{in }\,\, C([0,\tilde{\mathfrak t}];W^{-4,2})\cap C_w([0,\Tilde{\mathfrak t}]; L^{\gamma}(\T)),\\
        {S}_{h_m} \to \Tilde{S} \,\, \text{in }\,\, L^q(0,T;W^{-\mathfrak p-4,2}(\T)),\\
       {S}_{h_m} \rightharpoonup^\ast \Tilde{S} \,\, \text{in }\,\,  L^\infty(0,T; L^{\gamma}(\T)),\\
        {\vec{m}}^{h_m} \to \Tilde{\vec m} \, \, \text{in}\, \, C([0,\tilde{\mathfrak t}];W^{-4,2}(\T))\cap C_w ([0,\Tilde{\mathfrak t}];L^{\frac{2\gamma}{\gamma +1}}(\T)),\\
        \varrho^{h_m}\vartheta^{h_m} \to \Tilde{\varrho}\Tilde{\vartheta}+\Tilde{\mathcal{R}}_{\text{press}} \,\, \text{in} \,\, L_{w^*}^{\infty}(0,T;\mathcal{M}^{+}(\T,\R)),\\
        \frac{{\vec m}^h\otimes {\vec m}^h}{\varrho^h} \to \frac{\Tilde{\vec m}^h\otimes \Tilde{\vec m}^h}{\Tilde\varrho^h}+\Tilde{\mathcal{R}}_{\text{conv}} \,\, \text{in} \,\, L_{w^*}^{\infty}(0,T;\mathcal{M}^{+}(\T,\R^{3\times 3})), \\
\delta_{(\varrho^h,\vec m^h,S^h)}\to \Tilde{\mathcal{V}}_{t,x} \,\, \text{in} \,\, L_{w^*}^{\infty}(Q;\mathcal{P}(A)), 
        \end{cases}
    \end{equation}
in law as $h\rightarrow0$, where, for a filtered probability space $(\Tilde{\Omega},\Tilde\FF, (\Tilde\FF_t)_{t\geq 0},\Tilde\p)$, an $(\Tilde\FF_t)$-adapted Wiener process $\Tilde W$ and an $(\Tilde\FF_t)$-stopping time $\Tilde{\mathfrak t}$,
\[
((\Tilde{\Omega},\Tilde\FF, (\Tilde\FF_t)_{t\geq 0},\Tilde\p),\Tilde\varrho,\Tilde{\vec m},\Tilde S, \Tilde{\mathcal{R}}_{\text{conv}},\Tilde{\mathcal{R}}_{\text{press}},\Tilde{\mathcal{V}}_{t,x},\Tilde{\mathfrak t}, \Tilde{W})
\]
is a dissipative measure-valued solution to (\ref{Euler}) in the sense of Definition \ref{E:dfn} with initial law $\delta_{\varrho_0}\otimes\delta_{\vec m_0}\otimes\delta_{S_0}$. It holds $\tilde{\mathfrak t}\sim \mathfrak t$ in law.
\end{thm}
\begin{rmk}
A general framework for the numerical approximation of martingale solutions to nonlinear stochastic PDEs is given in \cite{OPW}. However, it is rather designed for parabolic problems and does not include the possibility for the solution to develop defect measures which is crucial for hyperbolic problems as considered here.
\end{rmk}
If we have a strong pathwise solution (as described in Definition \ref{def:compstrong}) we can obtain a stronger result. The strong solution exists on a given stochastic basis up to the stopping time $\mathfrak s$. For the strong solution, there is no vaccum and the energy is bounded, i.e., a counterpart of \eqref{ass:main} is satisfied. Hence one may hope that $\mathfrak s$ is significantly smaller than $\mathfrak t$. For $M\in\N$
we introduce the stopping time
\[
\mathfrak s_{M}=\inf\bigg \{t\in (0,\mathfrak{s})|\quad\|(r,\Theta,\vec v)(s,\cdot)\|_{W^{\mathfrak p+1,\infty}(\T)}>M\bigg\},
\]
where $\mathfrak p$ is the degree of the polynomials in the DG solution space, cf. Section \ref{sec:DG}.
Since $[r,\Theta,\vec v]$ is a strong solution,
\[
\p\left[\lim_{M\to \infty}\mathfrak s_M=\mathfrak{s}\right]=1.
\]
Our main result regarding convergence is the following theorem.
\begin{thm}\label{thm:main2}
Let $(\Omega, \FF, (\FF_t)_{t\geq 0},\p)$ be a complete stochastic basis with a probability measure $\p$ on $(\Omega,\FF)$ and right-continuous filtration $(\FF_t)_{t\geq 0}$, $W$ an $(\FF_t)$-adapted Wiener process and $\phi\in L_2(\mathfrak U;W^{\ell,2}(\mathcal O))$ for some $\ell\geq\max\{4,\mathfrak p+3\}$. Let  $\mathbf{U}^h=(\varrho^h,\vec m^h,\mathcal E^h)$ be the solution to \eqref{eq:DG_standard_2}. Suppose there is an $(\mathfrak F_t)$-stopping time $\mathfrak t$ such that \eqref{ass:main} holds.
Let the triplet $[r,\Theta,\vec v]$  and a stopping time $\mathfrak{s}$ be a strong solution in the sense of Definition \ref{def:compstrong} with initial data
\begin{align*}
\varrho_0,\vec m_0,\Theta_0\in W^{\mathfrak p+3,2}(\T),\quad\inf_{x\in\T}\varrho_0>K,\quad \sup_{x\in\T}\mathcal E_0<K.
\end{align*}
\begin{enumerate}
\item[(a)] Let $\mathfrak p=0$. Then it holds for all $t\in[0,T]$ and $M\in\N$
\begin{align}\label{eq:errorA}
\E \Big[\Big(\|\varrho^h-r\|^2_{L^2(\T)}+\|\vec u^h-\vec v\|^2_{L^2(\T)}+\|\vartheta^h-\Theta\|^2_{L^2(\T)}\Big)(t\wedge\mathfrak t\wedge\mathfrak s_M)\Big]\leq\,c h^{\frac{1}{2}}.
\end{align}
\item[(b)] For $\mathfrak p\geq1$, we further assume that there is some deterministic $\mathfrak c>0$ independent of $h$ such that
\begin{equation}\label{as_1}
|\mathbf{U}^h (\mathbf{x}_j ,t\wedge\mathfrak t) - \mathbf{U}^h (\mathbf{x}_i ,t\wedge \mathfrak t) | \leq \mathfrak c h \qquad \forall x_j, x_i \in \mathcal{K}.
\end{equation}
 Then, it holds for all $t\in[0,T]$ and $M\in\N$
\begin{align}\label{eq:errorB}
\E \Big[\Big(\|\varrho^h-r\|^2_{L^2(\T)}+\|\vec u^h-\vec v\|^2_{L^2(\T)}+\|\vartheta^h-\Theta\|^2_{L^2(\T)}\Big)(t\wedge\mathfrak t\wedge\mathfrak s_M)\Big]\leq\,c h.
\end{align}
\end{enumerate}
The constant $c$ above depends on $M$ and $K$ from \eqref{ass:main}. 
\end{thm}
\begin{rmk}
There is hardly any literature on the numerical approximation of local strong solutions to SPDEs. In fact, the only result which is comparable to Theorem \ref{thm:main2} we are aware of is \cite{BrDo}. In  \cite{BrDo} local strong solutions to the 3D incompressible stochastic Navier--Stokes equations are considered.
\end{rmk}
\begin{rmk}\label{rmk:noise}
In the case of a truly infinite dimensional noise, a practical implementation requires an approximation by a finite sum, i.e., replacing $W$ by $W^N=\sum_{k=1}^Ne_k\beta_k$ for some large $N\in\N$. This leads to the additional error term
\begin{align*}
\sum_{k=N+1}^\infty\int_0^{\tau\wedge\mathfrak s_m\wedge \mathfrak t}\int_{\T}\varrho^h|\phi e_k|^2\dx\dt\leq c(K)\sum_{k=N+1}^\infty\int_{\T}|\phi e_k|^2\dx
\end{align*}
in the proof of Theorem \ref{thm:main2}. Its size can be controlled by the choice of $N$.
\end{rmk}

\section{Consistency}\label{mvsproof}

The aim of this section is to prove Theorem \ref{thm:main1}. First of all, we can use the boundedness of $\mathcal E^h$ from \eqref{ass:main}. Arguing as in 
\cite{HoFeBr}
we can deduce the following bounds
\begin{align}\label{eq:2512}\begin{aligned}
    \varrho^h&\in L^{\infty}([0,\mathfrak t];L^{\gamma}(\T)), \\
    \vec m^h &\in L^{\infty}([0,\mathfrak t];L^{\frac{2\gamma}{\gamma +1}}(\T)),\\
    \frac{\vec m^h}{\sqrt{\varrho^h}}&\in L^{\infty}([0,\mathfrak t];L^{2}(\T)),\\
    \frac{\vec m^h \otimes \vec m^h}{\varrho^h}&\in L^{\infty}([0,\mathfrak t];L^{1}(\T)),\\
    S^h&\in L^{\infty}([0,\mathfrak t];L^{\gamma}(\T)),\\
    \frac{ S^h}{\sqrt{\varrho^h}}&\in L^{\infty}([0,\mathfrak t];L^{2\gamma}(\T)),
\end{aligned}
\end{align}
which hold $\p$-a.s. uniformly in $h$.

\subsection{The consistency formulation}
Due to our shape regular mesh, cf. \cite{KuLu}, and by following line by line the proof of Theorem 4.2 in \cite{LuOe} we obtain for
  all $\tau \in (0,T]$: 
 \begin{itemize}
  \item for all $\varphi \in C^{p+1}(\overline{\T})$:
  \begin{equation}\label{eq:consistency_rho}
 \left[  \int_{\T} \varrho^h \varphi \diff \bx \right]_{t=0}^{t=\tau} =\int_0^\tau \int_{\T} \bm^h \cdot \nabla_{\bx} \varphi \diff \bx \diff t
 +\int_0^\tau e_{\varrho^h} (t,\varphi) \diff t;
  \end{equation}

  \item for all $\boldsymbol{\varphi}\in C^{p+1}(\overline{\T};\R^2)$:
  \begin{equation}\label{eq:consistency_m}
\begin{aligned}
 \left[  \int_{\T} \bm^h \boldsymbol{\varphi} \diff \bx \right]_{t=0}^{t=\tau} &=\int_0^\tau \int_{\T} \Big(\frac{\bm^h\otimes\bm^h}{\varrho^h} :  \nabla_{\bx}\boldsymbol{\varphi}  +\varrho^h\vartheta^h \mathrm{div}_{\bx} \boldsymbol{\varphi} \Big)\diff \bx \diff t\\
&+\int_0^\tau\int_{\T} \varrho^h\phi\cdot\Pi_h\boldsymbol\varphi\diff \bx \,\dd W + \int_0^\tau e_{\bm^h} (t,\boldsymbol{\varphi}) \diff t,
\end{aligned}
  \end{equation}
  where $\Pi_h$ is the projection onto the solution space $\mathcal V^h$;
  \item for all $\psi \in C^{p+1}( \overline{\T}), \; \psi\geq 0$:
  \begin{equation}\label{eq:consistency_S}
 \left[  \int_{\T} S^h \psi \diff \bx \right]_{t=0}^{t=\tau} \leq \int_0^\tau \int_{\T} S^h \frac{\bm^h}{\varrho^h} \cdot \nabla_\bx {\psi} \diff \bx \diff t + \int_0^\tau e_{S^h} (t,\psi) \diff t;
  \end{equation}
  \item We have the energy balance
\begin{align}  \label{eq:consistency_E}
\int_{\T} \mathcal E^h(\tau)\diff \bx =\int_{\T} \mathcal E^h_0 \diff \bx+\frac{1}{2}\int_0^\tau\|\sqrt{\varrho^h}\phi\|_{L_2}^2 \,\dd t+ \int_0^\tau\int_{\T}\phi\cdot \vec m^h\diff\bx\,\dd W
\end{align}
  \item The error  $ \mathbf e^h:=(e_{j\varrho^h},e_{\vec m^h}, e_{S^h})$ tends to zero under mesh refinement: We have $\p$-a.s.
  \begin{equation}\label{eq:consistency_error}
   \norm{\mathbf e^h(\varphi,\boldsymbol{\varphi},\psi)}_{L^1(0,\mathfrak t)} \to 0 \text{ if } h\to 0. 
  \end{equation}
 \end{itemize}
Note that the proof of \eqref{eq:consistency_error} in \cite{LuOe} is only based on the bounds from \eqref{ass:main}, which is the reason why it only holds up to the stopping time in our case. We also remark that, in fact, a stronger statement is proved in \cite{LuOe} for the case $\mathfrak p\geq 1$ and under the additional assumption \eqref{as_1}. In our case it can be written as
   \begin{equation}\label{eq:consistency_errorB}
   \sup_{t\in[0,\mathfrak t]}|\mathbf e^h(\varphi,\boldsymbol{\varphi},\psi)|\leq\,ch\|(\varphi,\boldsymbol{\varphi},\psi)\|_{W^{\mathfrak p+1,\infty}(\T)},
  \end{equation}
where $c=c(\mathcal{K})$ with $\mathcal{K}$ from \eqref{ass:main}.

If $\mathfrak p=0$ instead, it is proved  using the Godunov/(local) Lax-Friedrichs fluxes that
   \begin{equation}\label{eq:consistency_errorA}
   \sup_{t\in[0,\mathfrak t]}|\mathbf e^h(\varphi,\boldsymbol{\varphi},\psi)|\leq\,ch^{\frac{1}{2}}\|(\varphi,\boldsymbol{\varphi},\psi)\|_{W^{1,\infty}(\T)}.
  \end{equation}
The latter is based on a weak BV-estimate.
\begin{rmk}
It should be stressed out that in the case $\mathfrak p=0$ we are in the classical FV context. The $L^1$ error in the relative energy with $1/2$ corresponds to $1/4$ in $L^2$ of the conserved (primitive) quantities.  Under the additional assumption that the total variation of the numerical solution is uniformly bounded  $1/2$ in the $L^2$-norm is achieved, cf. \cite{LuSh}.
The study in \cite{LuSh} employs the Godunov method. 
Nonetheless, the findings can be applied to the Lax-Friedrichs FV method by following the same analysis.
Moreover, it is important to highlight that, in the examination of consistency within the DG framework, the additional assumption \eqref{as_1} was introduced regarding the variation among the numerical solution at different nodal points within a single element and their proportional adjustments concerning the grid parameter $h$. The application of limiters was also disregarded. Including the application of limiters, a result similar to the FV setting can only be expected.\footnote{However, it is still unclear how limiters can be constructed and incorporated for the stochastic setting which will be part of future research.} This also is in line with the analysis of the flux corrected FE method approach using  monolithic convex limiting and residual distribution schemes. In these contexts, one can directly follow the analysis presented in \cite{LuSh}, expecting convergence rates of $1/4$ or $1/2$ only if the total variation of the numerical solution remains uniformly bounded.
Note that it is currently work in progress to develop new deterministic a priori error bounds for the DG method (as well as RD and MCL) by means of the consistency analysis.

\end{rmk}

 In order to prepare the proof of compactness we need some time-regularity of the variables $\varrho^h,\vec m^h$ and $S^h$.
 For the stochastic term in the momentum equation we have
 using continuity of $\Pi_h$
\begin{align*}
\E \left[ \left \|\int_{0}^{\cdot} \Pi_h(\varrho^h \phi )\,\dd W\right\|_{C^{\alpha}([0,T];L^2(\T))}^{q}\right] \leq\, c \,\E \left[ \int_{0}^{T}\Big\|\Pi_h(\sqrt{\varrho^h} \phi)\Big\|^{q}_{L_2(\UU,L^2(\T))} \, \dd t \right]
\end{align*}
 for all $\alpha \in (0,1/2-1/q)$ and $q>2$, see [\cite{Brt_1}, Lemma 9.1.3. b)] or [\cite{Hofm}, Lemma 4.6]). Using continuity of $\Pi_h$ we end up with
\begin{align*}
\E \left[ \left \|\int_{0}^{\cdot} \Pi_h(\varrho^h \phi) \,\dd W\right\|_{C^{\alpha}([0,T];L^2(\T))}^{q}\right] \leq\, c \,\E \left[ \int_{0}^{T}\Big\|\sqrt{\varrho^h} \phi\Big\|^{q}_{L_2(\UU,L^2(\T))} \, \dd t \right]  \leq\, c(\overline{\varrho},q,\phi,T).
\end{align*}
The final estimate follows from $\int_{\T}\varrho^h\dx=\int_{\T}\varrho_0\dx$ and \eqref{eq:HS}.

Now combining the deterministic estimates from \eqref{eq:2512} and \eqref{eq:consistency_errorB} with the stochastic estimate above, and using the embeddings $W^{1,2}_t \hookrightarrow C^{\alpha}_{t}$ and $L^2_x \hookrightarrow W^{-3,2}_x$ shows 
 \begin{equation}\label{estimate}
\E \left[\big\|\vec{m}^h\big\|_{C^{\alpha}([0,\mathfrak t];W^{-\mathfrak p-3,2}(\T))} \right] \leq c(T)
 \end{equation}
for all $\alpha < 1/2$ as a consequence of \eqref{eq:consistency_m}. Similarly, the continuity equation \eqref{eq:consistency_rho} together with \eqref{eq:2512} and \eqref{eq:consistency_errorB} yields $\partial_t \varrho_{\varepsilon} \in L^{\infty}(0,\mathfrak t;W^{-\mathfrak p-3,2}(\T))$ $\p$-a.s. In particular,  we obtain

\[
\E\left[\big\| \varrho^h\big\|_{C^{\alpha}([0,\mathfrak t];W^{-\mathfrak p-3,2}(\T))}\right] \leq C.
\]
Finally, we deduce
\begin{align} \label{eq:12.01}
\E\left[\|  S^h\|_{\mathrm{BV}([0,\mathfrak t];W^{-\mathfrak p-3,2}(\T))}\right] \leq C
\end{align}
from \eqref{eq:consistency_S}.

\subsection{Compactness Argument}

%
%

We set the spaces (where $q<\infty$ is arbitrary):

\begin{eqnarray*}
\mathscr{X}_{\Vec{m}}: = C([0,T];W^{-\mathfrak p-4,2}(\T))\cap C_w ([0,T];L^{\frac{2\gamma}{\gamma +1}}(\T)),&\quad& \mathscr{X}_{W} :=C([0,T];\UU_0),\\
\mathscr{X}_{\varrho}:=C([0,T];W^{-\mathfrak p-4,2}(\T))\cap C_w([0,T]; L^{\gamma}(\T)), &\quad& \mathscr{X}_{\text{C}}:=L^{\infty}_{w*}(0,T;\mathcal{M}(\T,\R^{3\times3})),\\
\mathscr{X}_{ S}:=L^q(0,T;W^{-\mathfrak p-4,2}(\T))\cap L^\infty_{w^\ast}(0,T; L^{\gamma}(\T)), &\qquad& \mathscr{X}_{\text{Q}}:=L^{\infty}_{w*}(0,T;\mathcal{M}(\T,\R^3)),\\
\mathscr X_{\mathfrak t}=\R,\qquad\mathscr{X}_{\text{P}}:=L^{\infty}_{w*}(0,T;\mathcal{M}^{+}(\T,\R)), &\qquad&
\mathscr{X}_{\text{e}}:=L^2(0,T;W^{-\mathfrak p-4}(\T)),
\end{eqnarray*}
 with respect to weak-* topology for all spaces with $L^{\infty}(\cdot,\mathcal{M}^{\cdot}(\cdot))$. Furthermore, for $T>0$ fixed, we choose the product path space
\begin{equation}
  \mathfrak{X}:=\mathscr{X}_{\varrho }\times \mathscr{X}_{ \Vec{m}} \times\mathscr{X}_{ S} \times\mathscr{X}_{\text{prss}} \times\mathscr{X}_{\text{conv}}\times\mathscr X_{\text{Q}}\times \mathscr{X}_{\text{e}}\times \mathscr{X}_{W}\times \mathscr{X}_{\mathfrak t},
\end{equation}
%
where $\mathscr{X}_{\text{prss}}$, $\mathscr{X}_{\text{conv}}$ and $\mathscr X_{\text{Q}}$ are the path spaces for the nonlinear terms
\[
 \mathrm{P}^h:=(\varrho^h)^{\gamma} \exp{\left(  \frac{ S^h}{c_v \varrho^h}\right)},\qquad\mathrm{C}^h:=\frac{\vec m^h\otimes \vec m^h}{\varrho^h},\qquad \mathrm Q^h:= S^h \frac{\vec m^h}{\varrho^h},
\]
respectively. 
We denote by $\mathscr B_{\mathfrak X}$ the Borelian $\sigma$-algebra on $\mathfrak X$ and by $$\mathcal{L}[\varrho^h(\cdot\wedge\mathfrak t),\vec{m}^h(\cdot\wedge\mathfrak t), \mathrm S^h(\cdot\wedge\mathfrak t), P^h(\cdot\wedge\mathfrak t), \mathrm C^h(\cdot\wedge\mathfrak t), \mathrm Q^h,\mathbf e^h(\cdot\wedge\mathfrak t), W,\mathfrak t]$$ the probability law on $\mathfrak{X}$. 
As in \cite{Mo} we can infer that it is a sequence of tight measures on $\mathfrak{X}$. Note that the variables $\mathfrak t$ and $\mathbf e^h$ are not included in \cite{Mo}. The law of $\mathfrak t$ is clearly tight as being a Radon measure on the Polish space $\R$, while tightness of $\mu_{ \mathbf{e}^h}$ follows from \eqref{eq:consistency_errorB}. Also, tightness of the law of $S^h$ follows from \eqref{eq:12.01} and Helly’s selection theorem.
In view of Jakubowksi's version of the Skorokhod representation theorem \cite{Jaku} (see also Brzezniak et al. \cite{Brzez}, and \cite[Section 2.8]{FrBrHo} for property c), we have the following proposition.

\begin{prop} \label{skorokhod}
There exists a nullsequence $(h_m)_{m \in \N}$, a complete probability space  $(\Tilde{\Omega}, \Tilde{\FF},\Tilde{\p})$ with $(\mathfrak{X},\mathscr{B}_{\mathfrak{X}})$-valued random variables  $(\Tilde{\varrho}^{h_m},\Tilde{\vec{m}}^{h_m},\Tilde{ S}^{h_m}, \Tilde{P}^{h_m}, \Tilde{C}^{h_m},\Tilde{Q}^{h_m},\Tilde{\mathbf e^{h_m}} ,\Tilde{W}^{h_m},\Tilde{\mathfrak t}^{h_m})$, $m \in \N$,\\ and $(\Tilde{\varrho},\Tilde{\vec{m}}, \Tilde{ S},\Tilde{P},\Tilde{C},\Tilde{Q},\tilde{\mathbf e}, \Tilde{W},\tilde{\mathfrak t})$ such that
\begin{itemize}
    \item [(a)] For all $m \in \N$ the law of $$(\Tilde{\varrho}^{h_m},\Tilde{\vec{m}}^{h_m},\Tilde{ S}^{h_m}, \Tilde{P}^{h_m}, \Tilde{C}^{h_m},\Tilde{Q}^{h_m},\Tilde{\mathbf e^{h_m}} ,\Tilde{W}^{h_m},\Tilde{\mathfrak t}^{h_m})$$ on $\mathfrak{X}$ coincides with $$\mathcal{L}[\varrho^{h_m}(\cdot\wedge\mathfrak t),\vec{m}^{h_m}(\cdot\wedge\mathfrak t), \mathrm S^{h_m}(\cdot\wedge\mathfrak t), P^{h_m}(\cdot\wedge\mathfrak t), \mathrm C^{h_m}(\cdot\wedge\mathfrak t), \mathrm Q^{h_m},\mathbf e^{h_m}(\cdot\wedge\mathfrak t), W,\mathfrak t];$$ 
    \item[(b)] The sequence 
$$(\Tilde{\varrho}^{h_m},\Tilde{\vec{m}}^{h_m},\Tilde{ S}^{h_m}, \Tilde{P}^{h_m}, \Tilde{C}^{h_m},\Tilde{Q}^{h_m},\Tilde{\mathbf e^{h_m}} ,\Tilde{W}^{h_m},\Tilde{\mathfrak t}^{h_m}), \,\,m \in \N,$$ converges $\Tilde{\p}$-almost surely to  $$(\Tilde{\varrho},\Tilde{\vec{m}}, \Tilde{ S},\Tilde{P},\Tilde{C},\Tilde{Q},\tilde{\mathbf e}, \Tilde{W},\tilde{\mathfrak t})$$ in the topology of $\mathfrak{X}$, i.e.
    \begin{equation}
        \begin{cases}
        \Tilde{\varrho}^{h_m} \to \Tilde{\varrho} \,\, \text{in }\,\, C([0,T];W^{-\mathfrak p-4,2})\cap C_w([0,T]; L^{\gamma}(\T)),\\
        \Tilde{S}^{h_m} \to \Tilde{S} \,\, \text{in }\,\, L^q(0,T;W^{-\mathfrak p-4,2}(\T)),\\
       \Tilde{S}^{h_m} \rightharpoonup^\ast \Tilde{S} \,\, \text{in }\,\,  L^\infty(0,T; L^{\gamma}(\T)),\\
       \Tilde{\vec{m}}_{h_m} \to \Tilde{\vec m} \, \, \text{in}\, \, C([0,T];W^{-\mathfrak p-4,2}(\T))\cap C_w ([0,T];L^{\frac{2\gamma}{\gamma +1}}(\T)),\\
        \Tilde{P}^{h_m} \to \overline{\Tilde{P}} \,\, \text{in} \,\, L_{w^*}^{\infty}(0,T;\mathcal{M}^{+}(\T,\R)),\\
        \Tilde{C}^{h_m} \to \overline{\Tilde{C}} \,\, \text{in} \,\, L_{w^*}^{\infty}(0,T;\mathcal{M}^{+}(\T,\R^{3\times 3})), \\
        \Tilde{Q}^{h_m} \to \overline{\Tilde{Q}} \,\, \text{in} \,\, L_{w^*}^{\infty}(0,T;\mathcal{M}(\T,\R^3)), \\
       \Tilde{\vec{e}}^{h_m} \to 0 \, \, \text{in}\, \, L^2(0,T;W^{-\mathfrak p-4,2}(\T)),\\
         \Tilde{W}^{h_m} \to \Tilde{W} \,\, \text{in} \,\, C([0,T];\UU_0),\\
         \Tilde{\mathfrak t}^{h_m} \to \Tilde{\mathfrak t} \,\, \text{in} \,\, \R,
        \end{cases}
    \end{equation}
    $\Tilde{\p}$-a.s.
    \item[(c)] For any Carath\'eodory function $H =H(t,x,\varrho,\vec m, S)$ where $(t,x)\in (0,T)\times\T,\, \, (\varrho,\vec m,S)\in \R^5$, satisfying for some $q_1,q_2,q_3>0$ the growth condition
    \[
    H(t,x,\varrho,\vec m, S)\lesssim 1 +|\varrho|^{q_1}+|\vec m|^{q_2}+|S|^{q_2}
    \]
    uniformly in $(t,x)$, we denote by $\overline{H(\varrho,\vec m, S)}(t,x) =\langle \mathcal{V}_{t,x},H \rangle$. Then it holds
    \[
    H(\tilde{\varrho}^{h_m},\tilde{\vec m}^{h_m}, \tilde{S}^{h_m}) \rightharpoonup \overline{H(\tilde{\varrho},\tilde{\vec m}, \tilde{S})}\quad\text{in}\,\,L^{k}((0,T)\times\T)
    \]
    $\tilde{\p}$-a.s. as  $m \to \infty$ for all $1 < k\leq\frac{\gamma+1}{q_1}\wedge \frac{2}{q_2}$.
    
\end{itemize}
\end{prop}

To guarantee adaptedness of random variables and to ensure that the stochastic integral is well-defined on the new probability space we introduce filtrations for correct measurability. We simplify notation and set
$ \Tilde{\mathcal{X}}:=\left[\tilde{\varrho},\tilde{\vec{m}},\tilde{ S}\right]$.
Let $\tilde{\FF}_t$ and $\tilde{\FF}_{t}^{h_m}$ be the $\tilde{\p}$-augmented filtration of the correspnding random variables from Proposition \ref{skorokhod}, i.e.
\begin{align*}
\tilde{\FF}_t &= \sigma (\sigma(t\vee\tilde{\mathfrak t},\rr \Tilde{\mathcal{X}},\rr\tilde{W}) \cup \sigma_t( \tilde{P},\tilde{C},\tilde{Q},\tilde e)\cup\{ \mathcal{N} \in \tilde{\FF};\tilde{\p}(\mathcal{N})=0\}), t\geq 0,\\
\tilde{\FF}_{t}^{\varepsilon_m}&=\sigma(\sigma(t\vee\Tilde{\mathfrak t},\rr\Tilde{\mathcal{X}}^{\varepsilon_m},\rr\tilde{W}^{h_m})\cup \sigma_t (\tilde{P}^{h_m},\tilde{C}^{h_m},\tilde{Q}^{h_m},\tilde{\mathbf e}^{h_m})\cup\{ \mathcal{N} \in \tilde{\FF};\tilde{\p}(\mathcal{N})=0\}), t\geq 0.
\end{align*}

Here $\rr$ denotes the restriction operator to the interval $[0,t]$ on the path space and $\sigma_t$\footnote{The family of $\sigma$-fields $(\sigma_{t}[\vec V])_{t\geq 0}$ given as random distribution history of 
 \begin{equation*}
     \sigma_t[\vec V]:= \bigcap_{s>t}\sigma\left(\bigcup_{\varphi\in C_c^{\infty}(Q;\R^3)}\{\langle \vec V, \varphi \rangle <1 \}\cup \{N\in \FF, \p(N)=0\} \right)
 \end{equation*}
 is called the history of $\vec V$. In fact, any random distribution is adapted to its history, see\cite{FrBrHo} (Chap. 2.2).} denotes the history of a random distribution.

\subsection{Passage to the limit}
Noticing that \eqref{eq:DG_standard_2} is a finite dimensional problem we can write in the form
\begin{align}\label{eq:1406a}
X_t^m&=X^m(0)+\int_0^t\mu(X^m_s)\,\dd s+\int_0^t\Sigma(X_s)\,\dd  W,
\end{align}
where $(X_t^m)$ is an $\R^N$-valued stochastic process (in, fact $X_t^m$ corresponds to $\bvu$ from \eqref{eq:DG_standard_2} with $h=h_m$), $\mu:\R^N\rightarrow\R^N$ and $\Sigma:\R^N\rightarrow\R^{N\times N}$ are Lipschitz continuous functions in the range of $(X^m_{t\wedge \mathfrak t})$ due to \eqref{ass:main}.
We have to show that $\tilde\p$-a.s.
\begin{align}\label{eq:1406b}
\tilde X_{t\wedge \tilde{\mathfrak t}}^m&=\tilde X^m(0)+\int_0^{t\wedge \mathfrak t}\mu(\tilde X^m_s)\,\dd s+\int_0^{t\wedge \mathfrak t}\Sigma(\tilde X^m_s)\,\dd \tilde W^m,
\end{align}
i.e., the equation continues to hold on the new probability space.  Here $\tilde X_t^m$ relates to $(\tilde\varrho^{h_m},\Tilde{\vec{m}}^{h_m},\Tilde{\mathcal E}^{h_m})$ exactly as $X_t^m$ relates to $(\varrho^{h_m},{\vec{m}}^{h_m},{\mathcal E}^{h_m})$ via \eqref{eq:DG_standard_2}.
By Proposition \ref{skorokhod} the processes $(X^m_{t\wedge\mathfrak t})$ and $(\tilde X^m_t)=(\tilde X^m_{t\wedge\tilde{\mathfrak t}})$ coincide in law on $C^0([0,T];\R^N)$ and the same holds true for $W$ and $\tilde W$.
The mapping $X\mapsto \int_0^{\cdot}\mu(t,X_t)\dt$ is continuous on the $C^0([0,T];\R^N)$. However,
the mapping $(X,W)\mapsto \int_s^t\Sigma(r,X_r)\,\diff W$ is not.
So, we can not identify it immediately. We will make use of the fact that a martingale is uniquely determined by its quadratic variations.
This can be done with the help of an
elementary method by Brzezniak and Ondrej\'at \cite{BZ}.
We consider the functionals
\begin{align*}
\mathfrak M(Y,\mathfrak r)_t&=Y_{t\wedge \mathfrak r}-Y(0)-\int_0^{t\wedge \mathfrak r}\mu(r,Y_r)\diff r,\\
\mathfrak N(Y,\mathfrak r)_t&=\int_0^{t\wedge \mathfrak r}|\Sigma(r,Y_r)|^2\diff r,\quad \mathfrak L(Y,\mathfrak r)_t=\int_0^{t\wedge \mathfrak r}\Sigma(r,Y_r)\diff r
\end{align*}
Obviously $\mathfrak M$, $\mathfrak N$ and $\mathfrak L$ are continuous on the pathspace. Consequently, by equality of laws, we have
\begin{align*}
\mathfrak M(X^m,\mathfrak t)_{t}&\sim^d \mathfrak M(\tilde X^m,\tilde{\mathfrak t})_{t},\quad
\mathfrak N(X^m,\mathfrak t)_{t} \sim^d \mathfrak N(\tilde X^m,\tilde{\mathfrak t})_{t},\quad
\mathfrak L(X^m,\mathfrak t)_{t}\sim^d \mathfrak L(\tilde X^m,\tilde{\mathfrak t})_{t}.
\end{align*}
Let $\mathfrak M(X^m,\mathfrak t)_{s,t}$ denote the increment $\mathfrak M(X^m,\mathfrak t)_{t}-\mathfrak M(X^m,\mathfrak t)_{s}$ and similarly for $\mathfrak N(X^m,\mathfrak t)_{s,t}$ and $\mathfrak N_k(X^m,\mathfrak t)_{s,t}$. 
Note that the proof will be complete once we show that the process $\mathfrak M(\tilde X^m, \tilde{\mathfrak t})_t$ is an $(\tilde{\FF}_t^m)_{t\geq0}$-martingale and its quadratic and cross variations satisfy, respectively,
\begin{equation}\label{marta}
\begin{split}
\langle\langle \mathfrak M(\tilde X^m,\tilde{\mathfrak t})\rangle\rangle_{t}&=\mathfrak N(\tilde X^m,\tilde{\mathfrak t})_{t},\quad\langle\langle \mathfrak M(\tilde X^m,\tilde{\mathfrak t}),\tilde W\rangle\rangle_{t}=\mathfrak L(\tilde X^m,\tilde{\mathfrak t})_{t}.
\end{split}
\end{equation}
Indeed, in that case we have
\begin{align}\label{neu3108a}
\Big\langle\Big\langle \mathfrak M(\tilde X^m,\tilde{\mathfrak t})-\int_0^{\cdot\wedge \tilde{\mathfrak t}}\Sigma(s,\tilde X^m_s)\,\dd\tilde W\Big\rangle\Big\rangle_{t}=0
\end{align}
which implies the desired equation \eqref{eq:1406b}. 
Let us verify \eqref{marta}. To this end, we fix times $s,t\in[0,T]$ such that $s<t$ and let
$$h:\R\times C^0([0,s];\R^N)\rightarrow [0,1]$$
be a continuous function.
Since $\mathfrak M(X^m,\mathfrak t)$
is a square integrable $(\F_t)_{t\geq0}$-martingale, we infer that
$$\big[\mathfrak M(X^m,\mathfrak t)\big]^2-\mathfrak N(X^m,\mathfrak t),\quad \mathfrak M(X^m,\mathfrak t)W-\mathfrak L(X^m,\mathfrak t),$$
are $(\FF_t)_{t\geq0}$-martingales.
Let $\bfr_s$ be the restriction of a function to the interval $[0,s]$. Then it follows from the equality of laws that
\begin{equation}\label{exp11}
\tilde{\E}\big[\,h\big(s\vee\tilde{\mathfrak t},\bfr_s\tilde X^m,\bfr_s\tilde{W}^m\big)\mathfrak M(\tilde{\mathfrak t},\tilde X^m,\tilde{\mathfrak t})_{s,t}\big]=\E \big[\,h\big(s\vee \mathfrak t,\bfr_sX^m,\bfr_sW\big)\mathfrak M(X^m,\mathfrak t)_{s,t}\big]=0,
\end{equation}
\begin{equation}\label{exp21}
\begin{split}
&\tilde{\E}\bigg[\,h\big(s\vee\tilde{\mathfrak t},\bfr_s\tilde X^m,\bfr_s\tilde{W}^m\big)\Big([\mathfrak M(\tilde X^m,\tilde{\mathfrak t})^2]_{s,t}-\mathfrak N(\tilde X^m,\tilde{\mathfrak t})_{s,t}\Big)\bigg]\\
&=\E\bigg[\,h\big(s\vee{\mathfrak t},\bfr_sX^m,\bfr_sW^m\big)\Big([\mathfrak M(X^m,\mathfrak t)^2]_{s,t}-\mathfrak N(X^m,\mathfrak t)_{s,t}\Big)\bigg]=0,
\end{split}
\end{equation}
\begin{equation}\label{exp31}
\begin{split}
&\tilde{\E}\bigg[\,h\big(s\vee\tilde{\mathfrak t},\bfr_s\tilde X^m,\bfr_s\tilde{W}^m\big)\Big([\mathfrak M(\tilde X^m,\tilde{\mathfrak t})\tilde{W}^m]_{s,t}-\mathfrak L(\tilde X^m,\tilde{\mathfrak t})_{s,t}\Big)\bigg]\\
&=\E\bigg[\,h\big(s\vee{\mathfrak t},\bfr_sX^m,\bfr_sW\big)\Big([\mathfrak M(X^m,\mathfrak t)W]_{s,t}-\mathfrak L(X^m,\mathfrak t)_{s,t}\Big)\bigg]=0.
\end{split}
\end{equation}
So we have shown \eqref{marta} and hence (\ref{neu3108a}). This finishes the proof of \eqref{eq:1406b} and thus have shown that the equations
  \begin{equation}\label{eq:consistency_rhotilde}
 \left[  \int_{\T} \tilde\varrho^h \varphi \diff \bx \right]_{t=0}^{t=\tau\wedge\tilde{\mathfrak t}} =\int_0^{\tau\wedge\tilde{\mathfrak t}} \int_{\T} \tilde{\bm}^h \cdot \nabla_{\bx} \varphi \diff \bx \diff t
 +\int_0^{\tau\wedge\tilde{\mathfrak t}} \tilde e_{\varrho^h} (t,\varphi) \diff t;
  \end{equation}
  \begin{equation}\label{eq:consistency_mtilde}
\begin{aligned}
 \left[  \int_{\T} \tilde{\bm}^h \boldsymbol{\varphi} \diff \bx \right]_{t=0}^{t=\tau\wedge\tilde{\mathfrak t}} &=\int_0^{\tau\wedge\tilde{\mathfrak t}} \int_{\T} \Big(\frac{\tilde{\bm}^h\otimes\tilde{\bm}^h}{\tilde\varrho^h} :  \nabla_{\bx}\boldsymbol{\varphi}  +\tilde\varrho^h\tilde\vartheta^h \mathrm{div}_{\bx} \boldsymbol{\varphi} \Big)\diff \bx \diff t\\
&+\int_0^{\tau\wedge\tilde{\mathfrak t}}\int_{\T} \tilde\varrho^h\phi\cdot\Pi_{h_m}\boldsymbol\varphi\diff \bx \,\dd \tilde W + \int_0^{\tau\wedge\tilde{\mathfrak t}} \tilde e_{\bm^h} (t,\boldsymbol{\varphi}) \diff t,
\end{aligned}
  \end{equation}
  \begin{equation}\label{eq:consistency_Stilde}
 \left[  \int_{\T} \tilde S^h \psi \diff \bx \right]_{t=0}^{t=\tau\wedge\tilde{\mathfrak t}} \leq \int_0^{\tau\wedge\tilde{\mathfrak t}} \int_{\T} \tilde S^h \frac{\tilde\bm^h}{\tilde\varrho^h} \cdot \nabla_\bx {\psi} \diff \bx \diff t + \int_0^{\tau\wedge\tilde{\mathfrak t}} \tilde e_{S^h} (t,\psi) \diff t;
\end{equation}
hold $\tilde\p$-a.s. for all sufficiently smooth spatial test-functions. For the stochastic integral in we have by It\^{o}-isometry, $\phi\in L_2(\mathfrak U;L^2(\T))$ and \eqref{ass:main}
\begin{align*}
\tilde\E\bigg[\bigg|\int_0^{\tau\wedge\tilde{\mathfrak t}}\int_{\T} \tilde\varrho^h\phi\cdot(\Pi_{h_m}-\mathrm{id})\boldsymbol\varphi\diff \bx \,\dd \tilde W\bigg|^2\bigg]&=\tilde\E\bigg[\int_0^{\tau\wedge\tilde{\mathfrak t}}\bigg(\int_{\T} \tilde\varrho^h\phi\cdot(\Pi_{h_m}-\mathrm{id})\boldsymbol\varphi\diff \bx\bigg)^2\dt \bigg]\\
&\leq\,c(\tau,K,\phi)\|(\Pi_{h_m}-\mathrm{id})\boldsymbol\varphi\|^2_{L^2(\T)},
\end{align*}
which vanishes as $m\rightarrow\infty$.
By Proposition \ref{skorokhod} we can pass to the limit in the remaining terms noticing that the error terms on the right-hand side vanish as $m\rightarrow\infty$. Hence the limit object is a dissipative solution in the sense of Definition \ref{E:dfn} provided we can establish the energy balance.

On the original probability space, the approximate the energy equality holds in the sense that
\begin{eqnarray*}
\mathfrak E_{t\wedge\mathfrak t}^{h_m}= \mathfrak E_{s\wedge\mathfrak t}^{h_m}+\frac{1}{2}\int_{s\wedge\mathfrak t}^{t\wedge\mathfrak t}\|\sqrt{\varrho^{h_m}}\phi\|_{L_2(\UU,L^2(\T))}^{2}\,\dd \sigma+ \int_{s\wedge\mathfrak t}^{t\wedge\mathfrak t}\int_{\T}\vec m^{h_m}\phi \,\dd x\dd W,
\end{eqnarray*}
$\p$-a.s for a.a $0\leq s<t$ (including $s=0$), where

\[
\mathfrak E_{t}^{h_m} =\int_{\T}\left[\frac{1}{2}\frac{|\vec m^{h_m}|^2}{\varrho^{h_m}}+ c_v(\varrho^{h_m})^{\gamma}\exp{\left(\frac{S^{h_m}}{c_v\varrho^{h_m}}\right)} \right]\, \dd x ,
\]
for a.a $t \geq 0$. For any fixed $s$ this is equivalent to

\[
-\int_{s\wedge\mathfrak t}^{T\wedge\mathfrak t}\partial_t \varphi \mathfrak E_{t\wedge\mathfrak t}^{h_m} \, \dd t- \varphi(s)\mathfrak E_{s\wedge\mathfrak t}^{h_m} = \frac{1}{2}\int_{s\wedge\mathfrak t}^{T\wedge\mathfrak t}\varphi \|\sqrt{\varrho^{h_m}}\phi\|_{L_2(\UU,L^2(\T))}^{2}\,\dd t+ \int_{s\wedge\mathfrak t}^{T\wedge\mathfrak t}\varphi\int_{\T}\vec m^{h_m} \cdot \phi \,\dd x\dd W,
\]
$\p$-a.s for all $\varphi\in C_{0}^{\infty}([s,T))$. By virtue of Theorem 2.9.1 in \cite{FrBrHo}, and in view of Proposition \ref{skorokhod} the energy equality continues to hold on the new probability space and reads

\[
\tilde{\mathfrak E}_{t\wedge\tilde{\mathfrak t}}^{h_m}=\tilde{\mathfrak E}_{s\wedge\tilde{\mathfrak t}}^{h_m}+\frac{1}{2}\int_{s\tilde{\wedge\mathfrak t}}^{t\wedge\tilde{\mathfrak t}}\|\sqrt{\tilde{\varrho}^{h_m}}\phi\|_{L_2(\UU,L^2(\T))}^{2}\,\dd \sigma+ \int_{s\wedge\tilde{\mathfrak t}}^{t\wedge\tilde{\mathfrak t}}\int_{\T}\tilde{\vec m}^{h_m}\phi \,\dd x\dd \tilde{W}^{h_m},
\]
$\tilde{\p}$-a.s. for a.a $s$(including $s=0$) and all $t \geq s$. Averaging in $t$ and s, the above expression becomes continuous on the path space.  Furthermore, fixing $s=0$ and Lemma 2.1 in \cite {Debussche}, the bounds established in Proposition \ref{skorokhod}, and higher moments to perform the limit $m\to \infty$ we obtain

\begin{equation}\label{eq:newEner}
    \tilde{\mathfrak{E}}_{t\wedge\tilde{\mathfrak t}}\leq\tilde{\mathfrak{E}}_0+\frac{1}{2}\int_{s\wedge\tilde{\mathfrak t}}^{t\wedge\tilde{\mathfrak t}}\|\sqrt{\tilde{\varrho}}\phi\|_{L_2(\UU;L^2(\T))}^2\, \dd \sigma + \int_{s\wedge\tilde{\mathfrak t}}^{t\wedge\tilde{\mathfrak t}}\int_{\T}\tilde{\vec m} \cdot \phi \, \dd x \dd \tilde{W},
\end{equation}
    $\p$-a.s. for a.a. $t\in [0,T]$, where
    \[
    \tilde{\mathfrak {E}}_t= \int_{\T}\left[\frac{1}{2}\frac{|\tilde{\vec m} |^2}{\tilde{\varrho}}+c_v\tilde{\varrho}^{\gamma}\exp{\left(\frac{\tilde{S}}{c_v\tilde{\varrho}}\right)}\right]\, \dd x + \frac{1}{2}\int_{\T}\dd\, \mathrm{tr}\tilde{\mathcal{R}}_{\text{conv}}(t) +c_v\int_{\T}\dd\, \tilde{\mathcal{R}}_{\text{press}}(t),
    \]
     and 
    \[
   \tilde{\mathfrak E}_0= \int_{\T}\left[\frac{1}{2}\frac{|{\vec m}_0 |^2}{{\varrho}_0}+c_v{\varrho}_0^{\gamma}\exp{\left(\frac{S_0}{c_v\varrho_0}\right)}\right]\, \dd x.
    \]
Performing the limit $\varepsilon_m\to 0$ yields an energy inequality. 
Hence, to convert (\ref{eq:newEner}) to equality, it is sufficient to augment  the term contributing to the internal energy ($\tilde{\mathcal{R}}_{\text{press}}(t)$) by $h(t)\dd x$ with A suitable spatially homogeneous $h\geq 0$. NOte that $\tilde{\mathcal{R}}_{\text{press}}(t)$ acts on $\mathrm{div}_x\varphi$ such that
$\int_{\T}h(t)\mathrm{div}_x \varphi\,\dd x=0$
and hence
 \[
-\int_{s\wedge\tilde{\mathfrak t}}^{T\wedge\tilde{\mathfrak t}}\partial_t \varphi \tilde{\mathfrak E}_t \, \dd t- \varphi(s)\tilde{\mathfrak E}_{s\wedge\tilde{\mathfrak t}} = \frac{1}{2}\int_{s\wedge\tilde{\mathfrak t}}^{T\wedge\tilde{\mathfrak t}}\varphi \|\sqrt{\tilde\varrho}\phi\|_{L_2(\UU,L^2(\T))}^{2}\,\dd t+ \int_{s\wedge\tilde{\mathfrak t}}^{T\wedge\tilde{\mathfrak t}}\varphi\int_{\T}\tilde{\vec m}\cdot \phi \,\dd x\dd \tilde W,
\]
$\p$-a.s for all $\varphi\in C_{0}^{\infty}([s,T))$. This completes the proof of Theorem \ref{thm:main1}. 

\section{Convergence}\label{se_Convergence}
In this section we prove Theorem \ref{thm:main2}. The main tool is the relative entropy, which we derive in the next subsection.

\subsection{The relative entropy}
We consider the \textit{relative entropy} for the discrete solution $(\varrho^h,\vec m^h, \mathcal E^h)$ measuring its distance to a smooth comparison function $(r,\Theta,\vec v)$ (which will be chosen in the next subsection as the strong pathwise solution to \eqref{Euler}). It is given by
\begin{align}\label{relative}
\begin{aligned}
\mathcal{K}\big(\varrho^h,E^h,\vec m^h\big|r,\Theta,\vec v\big) &= \frac{1}{2}\int_{\T}\varrho^h\left|\frac{\vec m^h}{\varrho^h} -\vec v\right|^2\, \dd x-\int_{\T}\Theta\varrho^h {s}(\varrho^h,\mathcal E^h)\, \dd x-\int_{\T}\varrho^h\partial_{\varrho}H_{\Theta}(r,\Theta)\,\dd x\\&+\int_{\T}\big(\partial_{\varrho}H_{\Theta}(r,\Theta)(r)-H_{\Theta}(r,\Theta)\big)\, \dd x
\end{aligned}
\end{align}
for $t\in[0,\mathfrak t]$. Here we used
the \textit{ballistic free energy}
\begin{equation}\label{ballistic}
 H_{{\Theta}}(\varrho,\vartheta)=\varrho e(\varrho,\vartheta)-\Theta \varrho s (\varrho,\vartheta),
\end{equation}
recalling the formulas for $e$ and $s$ from Section \ref{sec:cr}.
We obtain the following result.
\begin{prop}[Relative Entropy Inequality]\label{propA}
Let $(\Omega, \FF, (\FF_t)_{t\geq 0},\p)$ be a complete stochastic basis with a probability measure $\p$ on $(\Omega,\FF)$ and right-continuous filtration $(\FF_t)_{t\geq 0}$ and $W$ an $(\FF_t)$-adapted Wiener process. Let  $(\varrho^h,\vec m^h,\mathcal E^h)$ be the solution to \eqref{eq:DG_standard_2}. Suppose there is an $(\mathfrak F_t)$-stopping time $\mathfrak t$ such that \eqref{ass:main} holds. Let $(r,\Theta, \vec v)$ be a trio of $(\FF_t)_{t\geq 0}$ -adapted stochastic processes defined on $(\Omega, \FF, (\FF_t)_{t\geq 0},\p)$ such that for some $T>0$
\begin{align*}
 r,\Theta,\vec v \in C([0,T]; C^{1} (\T)) \quad\p \text{-a.s.,}  \,\,\,
 \E\left[\sup_{t\in [0,T]}\|(r,\Theta,\vec v)\|_{W^{1,\infty}(\T)}\right]^q < \infty \text{ for all }\, 2\leq q < \infty,
\end{align*}
\[
0<\underline{r} \leq r(t,x) \leq \overline{r}, \quad
0<\underline{\Theta} \leq \Theta(t,x) \leq \overline{\Theta} \quad \p\text{-a.s.}
\]
Suppose further that it holds
\begin{align}\label{eq:sass}
\begin{aligned}
    \dd r &= D_t^dr\,\dd t,\\
    \dd \vec v &=D_t^d \vec v \,\dd t+ \mathbb{D}_t^s \vec v\, \dd W,\\
    \dd \Theta &= D_t^d \Theta\,\dd t,
\end{aligned}
\end{align}
where
\[
D^dr,D^d \Theta,D^d\vec v \in L^{q}(\Omega; C([0,T];C^{1}(\T)))\qquad  \mathbb{D}^s\vec v \in L^2(\Omega; L^2(0,T;L_2(\UU;L^2(\T))),
\]
\begin{equation}\label{eq:prop}
 \left(\sum_{k\geq 1}|\mathbb{D}^s\vec v(e_k)|^q\right)^{\frac{1}{q}}   \in L^{q}(\Omega; L^q(0,T;L^q(\T))).
\end{equation}
Then the relative entropy inequality:
\begin{align}\label{REI}
\begin{aligned}
   \mathcal{K}\bigg(\varrho^h,\mathcal E^h,\vec m^h\bigg|r,\Theta,\vec v\bigg)&\leq\,c \,\mathcal{K}\bigg(\varrho^h,\mathcal E^h,\vec m^h\bigg|r,\Theta,\vec v\bigg)(0)+c\int_{0}^{\tau}\mathcal{Q}\bigg(\varrho^h,\vartheta^h,\vec m^h\bigg|r,\Theta,\vec v\bigg)\, \dd t+\mathbb{M}\\
&+c\,\mathbf e^h(\partial_{\varrho}H_{\Theta}(r,\Theta),\vec v,\Theta)
\end{aligned}
\end{align}
holds $\p$-a.s for all $\tau \in (0,T\wedge\mathfrak t)$, where
\begin{align*}
\mathcal{Q}\bigg(\varrho^h,\vartheta^h,\vec m^h\bigg|r,\Theta,\vec v\bigg)=
&\int_{\T}\varrho\left(\frac{\vec m^h}{\varrho^h}-\vec v\right)\cdot\nabla_x\vec v\cdot \left(\vec v-\frac{\vec m^h}{\varrho^h}\right)\,\dd x\nonumber\\
&+\int_{\T}[(D_t^d \vec v + \vec v\cdot\nabla_x \vec v )\cdot(\varrho\vec v- \vec m^h) -p(\varrho^h,\vartheta^h)\mathrm{div}_x\vec v] \, \dd x  \\
&-\int_{0}^{\tau}\int_{\T}[\varrho^h s(\varrho^h,\vartheta^h) D_t^d \Theta + s(\varrho^h,\vartheta^h)\vec m^h\cdot \nabla_x \Theta ] \, \dd x \dd t\nonumber\\
&+\int_{0}^{T}\int_{\T}[\varrho^h s (r,\Theta) \partial_{t}\Theta+ \vec m^h s(r,\Theta)\cdot\nabla_x\Theta]\, \dd x\dd t \nonumber\\
&+\int_{\T} \left(\left(1-\frac{\varrho^h}{r}\right)\partial_{t}p(r,\Theta)-\frac{\vec m^h}{r}\cdot\nabla_x p(r,\Theta)\right)\,\dd x\nonumber,\\
&-\sum_{k\geq 1}\int_{\T}\mathbb{D}_t^s\vec v(e_k)\cdot \varrho^h\phi(e_k)\, \dd x \\
&+\frac{1}{2}\|\sqrt{{\varrho^h}}\phi\|_{L_2(\UU,L^2(\T))}^{2} +\frac{1}{2}\sum_{k\geq 1}\int_{\T}\varrho^h|\mathbb{D}_t^s\vec v(e_k)|^2\, \dd x, 
\end{align*}
and
\begin{align*}
\mathbb{M}&= \int_{0}^{\tau}\int_{\T}{\vec m^h}\phi \,\dd x\dd {W} \\
& -\int_{0}^{t}\int_{\T}\bigg[\vec m^h\mathbb{D}_t^s\vec v+ \Pi_h\vec v\varrho^h \phi\bigg]\,\dd x \dd W +\int_{0}^{t}\int_{\T}\varrho^h\vec v\cdot \mathbb{D}_t^s\vec v\, \dd x\dd W.
\end{align*}
\end{prop}

\begin{proof}
By It\^{o}'s formula we obtain from \eqref{eq:consistency_m} and \eqref{eq:sass}
\begin{align}\label{eq:PA}
\begin{aligned}
\dd \left( \int_{\T}\vec m^h\cdot\vec v \,\dd x \right)&=\left(\int_{\T}\bigg[\vec m^h\cdot D_t^d\vec v+ \left(\frac{\vec m^h\otimes\vec m^h}{\varrho^h}\right)\cdot\nabla \vec v+ p(\varrho^h, s^h)\mathrm{div}\vec v\bigg]\,\dd x \right)\, \dd t\\
&+\sum_{k\geq 1}\int_{\T}\mathbb{D}_t^s\vec v(e_k)\cdot \varrho^h\phi(e_k)\, \dd x \dd t\\
&+e_{\vec m^h}(t,\vec v)\dt+\dd M_1,\end{aligned}
\end{align}
where 
\[
M_1 =\int_{0}^{t}\int_{\T}\bigg[\vec m\mathbb{D}_t^s\vec v+ \Pi_h\vec v\varrho \phi\bigg]\,\dd x \dd W.
\]

Similarly to (\ref{eq:PA}), we compute
\begin{align}\label{eq:PB}
\begin{aligned}
\dd\left(\int_{\T}\frac{1}{2}\varrho^h |\vec v|^2\,\dd x\right)&=\int_{\T}\varrho^h\vec v \cdot \nabla \vec v\cdot \vec v\, \dd x \dd t+ \int_{\T}\varrho^h\vec v\cdot D_t^d\vec v\, \dd x \dd t\\
&+\frac{1}{2}\sum_{k\geq 1}\varrho^h|\mathbb{D}_t^s\vec v(e_k)|^2\, \dd x \dd t +\dd M_2, \end{aligned}
\end{align}
where 
\[
M_2 = \int_{0}^{t}\int_{\T}\varrho^h\vec v\cdot \mathbb{D}_t^s\vec v\, \dd x\dd W.
\]

Testing the entropy balance \eqref{eq:consistency_S} with $\Theta$ we deduce
\begin{equation}\label{eq:PD}
\dd \left(\int_{\T} \varrho^h s(\varrho^h,\vartheta^h) \Theta \dd x\right) 
\geq\int_{\T} \vec m^h s(\varrho^h,\vartheta^h) \cdot \nabla_x \Theta\, \dd x\dd t+\int_{\T}\varrho^h s(\varrho^h,\vartheta^h) \partial_t \Theta\, \dd x\dd t+e_{S^h}(t,\Theta)\dt.
\end{equation}
Similarly, testing the continuity equation \eqref{eq:consistency_rho} with $\partial_{\varrho}H_{\Theta}(r,\Theta)$ yields
\begin{equation}\label{eq:PF}
 \dd\left(\int_{\T}\varrho \partial_{\varrho}H_{\Theta}(r,\Theta)\, \dd x \right) =\int_{\T} \vec m\cdot \nabla_x(\partial_{\varrho}H_{\Theta}(r,\Theta)) \,\dd x+ \int_{\T}\varrho^h \partial_t(\partial_{\varrho}H_{\Theta}(r,\Theta))\, \dd x+e_{\varrho^h}(t,\partial_{\varrho}H_{\Theta}(r,\Theta))\dt.
\end{equation}
Finally, we collect and sum the resulting expressions (\ref{eq:PA})--(\ref{eq:PF}), and add the energy balance from \eqref{eq:consistency_S} to the sum.
\end{proof}

\subsection{Error estimates}
\begin{proof}[Proof of Theorem \ref{thm:main2}]
We choose now $(r,\vec v,\Theta)$ in Proposition \ref{propA} to be the local strong solution from Definition and $t\in[0,\mathfrak t\wedge \mathfrak s_M]$. The terms contained in $\mathcal Q$ corresponding exactly to those in \cite{Mo}\footnote{In fact, in \cite{Mo} also the defect measures appear which are not present for the discrete solution.} and we obtain as there
\begin{align*}
\mathcal{Q}\bigg(\varrho^h,\vartheta^h,\vec m^h\bigg|r,\Theta,\vec v\bigg)(\tau\wedge \mathfrak s_M\wedge \mathfrak t)\leq\,c  \int_0^{\tau\wedge \mathfrak s_M\wedge \mathfrak t} \mathcal{K}\bigg(\varrho^h,\mathcal E^h,\vec m^h\bigg|r,\Theta,\vec v\bigg)\,\diff t
\end{align*}
$\p$-a.s. with a constant depending on $M$.
The right-hand side can be handled by Gronwall's lemma. The stochastic terms in \eqref{REI} vanish after taking expectations. Finally, due to the regularity of $(r,\vec v,\Theta)$ we can use estimate \eqref{eq:consistency_errorB} and \eqref{eq:consistency_errorA} (depending on whether $\mathfrak p\geq 1$ or $\mathfrak p=0$) to control the error term by $h$. By \cite[Lemma 2.7]{LuSh} we have for some $c>0$ independent of $h$
\begin{align*}
&\E\Big[\mathcal{K}\bigg(\varrho^h,\mathcal E^h,\vec m^h\bigg|r,\Theta,\vec v\bigg)(\tau\wedge \mathfrak s_M\wedge \mathfrak t)\Big]\\&\qquad\qquad\leq\,c\,\E \Big[\Big(\|\varrho^h-r\|^2_{L^2(\T)}+\|\vec u^h-\vec v\|^2_{L^2(\T)}+\|\vartheta^h-\Theta\|^2_{L^2(\T)}\Big)(t\wedge\mathfrak t\wedge\mathfrak s_M)\Big],
\end{align*}
where the constant depends on $K$ from \eqref{ass:main}. The proof is complete.
\end{proof}

\section{Numerical simulation}\label{sec_numerical}
To support our theoretical analysis, we derive some simulations solving the stochastic Euler system \eqref{EulerC}. The numerical experiments are conducted using the Julia programming language \cite{bezanson2017julia}.  
For stochastic time integration we employ if nothing else is said the Euler–Maruyama method from the DifferentialEquations.jl package \cite{rackauckas2017differentialequations}. 
Spatial discretization is handled using the Trixi.jl framework \cite{ranocha2022adaptive,schlottkelakemper2021purely}.
Visualization of the results is achieved through Plots.jl \cite{christ2023plots}, and Random.jl package of Julia is used for generating the noise environment.  
As described before, up-to-the-authors knowledge this is the first time to present  simulations for this type of problems where discontinuities may rise and we lose regularity inside the path even if we start with smooth initial data. 
In our numerical testing, we focus therefore either on short time periods and/or weak noise strength meaning that if we apply \eqref{eq:14.02}, we multiply with constants which we choose appropriately to our testing. These numerical simulations should only give a proof of concept where further investigations  in this direction together with a more efficient implementation strategy is part of future research. Meanwhile we also refer to \cite{breit2023mean} for a numerical investigation in context of the stochastic Navier-Stokes equations.   \\
We are investigating two errors in detail. We are investigating the expected value of the classical discrete $L^2$ norm in terms of density and the expected value \eqref{eq:errorA}  which represents the estimation of relative energy. The error behaviour in other quantities, such as momentum and energy, follows a similar pattern to that of the density in our observations and is therefore omitted from further analysis. We analyse therefore
\begin{align}
\mathtt{E}_1&=\E \left[ \|\varrho^h-\varrho\|_{L^2(\T)} \right]; 
  \label{L2_error}\\
\mathtt{E}_2&=\E\Big[\Big(\|\varrho^h-r\|^2_{L^2(\T)}+\|\vec u^h-\vec v\|^2_{L^2(\T)}+\|\vartheta^h-\Theta\|^2_{L^2(\T)}\Big) \Big]. \label{relative_energy}
\end{align}
 at the end time. 
Since an analytical solution is unavailable, we compute a reference solution on the finest grid using the same set-up and project the numerical solution $ \varrho$ onto this finer grid to evaluate the errors. For the spatial discretisation, we consistently use the local Lax-Friedrichs (LLF) or Rusanov flux for the surface flux, while for the volume flux, we apply the entropy-conservative and kinetic-energy-preserving flux \eqref{eq:Ranocha_flux} from \cite{ranocha2018}. With polynomial degree zero, this approach aligns with the classical cell-centered finite volume (FV) framework, where the FV method employs the LLF flux.
For the DG case, we restrict ourself in the following mainly to the polynomial degree one for simplicity and efficiency. 
As stated before, this part is only a proof of concept, a more and detailed  investigation will be left for future work together with additional tasks like the construction of limiting strategies for the stochastic time-integration method and  a more efficient implementation. In the simulations, we set the number of samples to 1000 and consider one- and two-dimensional cases.

\subsection{One dimension}\label{sec_one_d}

We consider  the one-dimensional stochastic compressible Euler equations $\gamma = 1.4$
\begin{equation}\label{eq:euler}
  \dd \begin{pmatrix} \varrho \\  m \\ \varrho e \end{pmatrix}
  + \dd_x \begin{pmatrix}  m \\  m^2/\varrho + p \\ (\varrho e + p) v \end{pmatrix} \dd t
  = \frac{1}{2}\mu^2 \begin{pmatrix} 0 \\ 0 \\ \varrho  \end{pmatrix}  \dd t +
   \mu \begin{pmatrix} 0 \\ \varrho  \\  m  \end{pmatrix}  \dd \beta
\end{equation}
where $\mu$ denotes  a stochastic forcing factor and $\beta$ is a standard one-dimensional Wiener process. If nothing else is said, we select $\mu=1$ in the one dimensional tests.

\subsubsection{Grid Convergence}\label{sub_sub_section_1}
We consider a simple density wave 
\begin{equation}
  \varrho(0, x) = 1 + 0.5 \sin\bigl(2\cdot x \bigr), \quad v(0, x) = 0.1, \quad p(0,x)= 10,
\end{equation}
für $t \in [0, 0.5]$ and $x \in [-1, 1]$. In the deterministic case, we obtain a simple density wave (advection behaviour). The reference solution is calculated using $2^{12}=4096$ elements and the time-step is set to $\Delta t=10^{-5}$ where $T=0.5$.\\
In Table \ref{FV_DG_density}, the errors are plotted.  We observe error behaviour for the FV methods in density around 0.8 and the estimation in relative entropy is around 1.23 which is a little bit better than from our theoretical result but beyond the accuracy estimation for classical FV methods which is one. \\
In the case of DGSEM with order $\mathfrak{p=1}$ we obtain second order for the density error (and all other quantities) and fourth order in terms of the estimation of relative energy. 
These results are better than the one obtained through our analysis which relies on the consistency investigation and does take the regularity of the solution into account. 
We stress that the results presented lies inside the behaviour which one expects for DGSEM methods applied to the deterministic Euler equations as demonstrated already in several works, for instance \cite{gassner2016,zbMATH07517718,chen2017}.


%
%
%

%
%
%

{\small{
\begin{table}[h!]
\centering
\begin{tabular}{|c|c|c|c|c|c|c|c|c|}
\hline
& \multicolumn{4}{c|}{FV} & \multicolumn{4}{c|}{DGSEM ($\mathfrak{p}=1$)} \\
\cline{2-9}
Elements & $\mathtt{E}_1$ & EOC($\mathtt{E}_1$) & $\mathtt{E}_2$ & EOC($\mathtt{E}_2$) & $\mathtt{E}_1$ & EOC($\mathtt{E}_1$) & $\mathtt{E}_2$ & EOC($\mathtt{E}_2$) \\
\hline
64 & 0.2492 & - & 13.6617 & -  &   0.0029& - & 0.0041 & -\\ 
128 & 0.1616 & 0.63&   7.4707 & 0.87&  0.0007 & 2.00 &  0.0003&3.98 \\
256 & 0.0913& 0.82 & 3.1981 & 1.23  & 0.0002 & 2.00 & $1.63\cdot 10^{-5}$ & 4.00  \\
512 & 0.0464  &  0.98 & 1.0426 & 1.60 & $4.48\cdot 10^{-5}$ & 2.00 & $1.01\cdot 10^{-6}$ &  4.02 \\ \hline
average EOC& - &  0.81 & - & 1.23 & -& 2.0 & -& 4.00 \\ \hline
\end{tabular}
\caption{LLF FV and  entropy stable DGSEM ($\mathfrak{p}=1$) methods - error plot using 1000 samples}
\label{FV_DG_density}
\end{table}
}}


\subsubsection{One-dimensional Riemann problem}
Next, we examine one-dimensional Riemann problems in the domain \( [0,1] \). It is important to note that our Theorem \ref{thm:main2} does not apply to these test cases, as the initial conditions do not satisfy the requirement of belonging to \( W^{\mathfrak{p}+3,2} \); in fact, they are no longer even Lipschitz continuous. These types of problems are included for comparison with the results from \cite{LuSh}, where the Godunov method and a particular FV method were numerically studied in the deterministic case. The reference solutions are again calculated using \(4096\) elements. The initial data are as follows:
\begin{itemize}
  \item Rarefaction wave: $  (\varrho, v, p) (0, x) = \begin{cases}
  (0.5197, - 0.7259,0.4), \qquad x <0.5,\\
  (1, 0, 1),\qquad x>0.5.
  \end{cases}$
\item Contact discontinuity: $  (\varrho, v, p) (0, x) = \begin{cases}
  (0.5,0.5,5), \qquad x <0.5,\\
  (1, 5,0, 5),\qquad x>0.5.
  \end{cases}$
    \item Shock wave: $  (\varrho, v, p) (0, x) = \begin{cases}
  (1,0.7276,1), \qquad x <0.5,\\
  (0.5313, 0, 0.4),\qquad x>0.5.
  \end{cases}$
\end{itemize}
We use  $\Delta t=10^{-5}$  and $T=0.2$ in all three test cases. The boundary conditions are outflow conditions. 
In Tables \ref{FV_DG_rarefaction} and \ref{FV_Contact_Shock}, 
the errors and experimental orders of convergence are given for different settings. 
To summarize the results: 

\begin{itemize}
\item Rarefaction wave: For $\mathfrak p=0$ we recognise that convergence rates in terms of density are slightly larger than $0.5$ (in average 0.62) whereas the 
estimation of relative energy is around 1.28 (larger than 1). This is in agreement with the deterministic case from \cite{LuSh} where for Godunov method similar results have been presented and verified the theoretical result\footnote{In \cite{LuSh}, the assumption on the initial data was less restrictive and only an estimation in terms of projection from the initial data and the initial data itself was used.}. In the DGSEM setting with $\mathfrak{p}=1$, we recognise that due to the lack of regularity and on account of the Gibbs phenomena (projection at the beginning of the simulation), the convergence error rates are lower in both quantities. Here, additional techniques have to be used like the application of modal filters, for instance similar to the works \cite{offner2015zweidimensionale,zbMATH00036793}, to obtain better error rates. In the following, we will concentrate on the FV setting. 
\item For the single contact wave the convergence rate for the density is in average around $0.43$ which is slightly better than $1/4$. In terms of the estimation of the relative energy we have $0.84$ where the most accurate results are obtained on the finest grid.
\item For the single shock wave the convergence rate of $\mathtt{E}_1$ is around $0.52$. The error in $\mathtt{E}_2$ is about $1.02$. This outcome is consistent with the deterministic case discussed in \cite{LuSh}.
\end{itemize}

{\small{
\begin{table}[h!]
\centering
\begin{tabular}{|c|c|c|c|c|c|c|c|c|}
\hline
& \multicolumn{4}{c|}{FV} & \multicolumn{4}{c|}{DGSEM ($\mathfrak{p}=1$)} \\
\cline{2-9}
Elements & $\mathtt{E}_1$ & EOC($\mathtt{E}_1$) & $\mathtt{E}_2$ & EOC($\mathtt{E}_2$) & $\mathtt{E}_1$ & EOC($\mathtt{E}_1$) & $\mathtt{E}_2$ & EOC($\mathtt{E}_2$) \\
\hline
64 & 0.0321 & - & 0.0046 & -               &  0.0141 & -         &  0.00111  & -\\ 
128 & 0.0224 & 0.51 & 0.0021 & 1.09 &   0.0098  & 0.52  & 0.00077  & 0.52\\
256 & 0.0148 & 0.61 & 0.0009 & 1.26  & 0.0074& 0.41     & 0.00066 & 0.23  \\
512 & 0.0089 & 0.72 & 0.0003 & 1.47  & 0.0061 & 0.27    &  0.00061 & 0.08 \\ \hline
average EOC& - & 0.61 & - & 1.27        & -          &    0.4 & -          & 0.28 \\ \hline
\end{tabular}
\caption{LLF-FV and  entr. dissipative DGSEM methods for the rarefaction wave - 1000 samples}
\label{FV_DG_rarefaction}
\end{table}
}}

%
%
%

%
%
%

%
%
%
%

{\small{
\begin{table}[h!]
\centering
\begin{tabular}{|c|c|c|c|c|c|c|c|c|}
\hline
 & \multicolumn{4}{c|}{Contact} & \multicolumn{4}{c|}{Shock} \\
\cline{2-9}
Elements & $\mathtt{E}_1$ & EOC($\mathtt{E}_1$) & $\mathtt{E}_2$ & EOC($\mathtt{E}_2$) & $\mathtt{E}_1$ & EOC($\mathtt{E}_1$) & $\mathtt{E}_2$ & EOC($\mathtt{E}_2$) \\\hline
64 & 0.0673 & - & 0.5913 & -  &  0.049 & - &  0.010 & -\\ 
128 & 0.0525 & 0.36& 0.3620 & 0.71 & 0.037 & 0.45 & 0.005 & 0.86 \\
256 & 0.0395 & 0.41 & 0.2058 &0.81  & 0.025 & 0.51 & 0.003 & 1.01  \\
512 & 0.0279 & 0.50 & 0.1031 & 1.00 & 0.017 & 0.60 & 0.001 & 1.19\\ \hline
average EOC& - & 0.42 &  & 0.84& -& 0.52 & -            & 1.02\\ \hline
\end{tabular}
\caption{LLF-FV method - 1000 samples}
\label{FV_Contact_Shock}
\end{table}
}}

\subsubsection{1D SOD}
This experiment is the famous SOD problem. The exact deterministic solutions contains a rarefaction, contact and shock waves. In this example, the final time is set to 
$T=0.15$ and the initial data  is given by 
\begin{equation*}
  (\varrho, v, p) (0, x) = \begin{cases}
  (1,0,1), \qquad x <0.5,\\
  (0.125,0, 0.1),\quad x>0.5.
  \end{cases}
\end{equation*}
We use again $4096$ elements and $\Delta t=10^{-5}$. The convergence rates can be seen in Table \ref{FV_SOD} and we recognize that even without the regulartiy of the intitial data, we obtain convergence rates which align with the result itself inside the FV setting ($\approx 0.5$ in  $\mathtt{E}_1$ and  $\approx 1$ in  $\mathtt{E}_2$). 
{\small{
\begin{table}[h!]
\centering
\begin{tabular}{|c|c|c|c|c|}
\hline
Elements & $\mathtt{E}_1$ & EOC($\mathtt{E}_1)$ &$\mathtt{E}_2$ &EOC($\mathtt{E}_2)$ \\ \hline \hline
64 & 0.0378 & - &  0.0203 & - \\ 
128 & 0.0287 & 0.40 & 0.0106 & 0.93 \\
256 & 0.0200 & 0.51 & 0.0052& 1.05 \\
512 & 0.0130 & 0.63 & 0.0022 & 1.20 \\ \hline
av. order & - & 0.51 & -&1.06 \\ \hline
\end{tabular}
\caption{FV method for SOD -  1000 samples}
\label{FV_SOD}
\end{table}

}}
We also illustrate the impact of the $\mu$ parameter and provide examples demonstrating various noise behaviours. Figure \ref{fig:SOD} presents the density profiles for different configurations of the (stochastic) Euler equations, with a zoomed-in view around the shock wave. The simulations utilise the FV framework with $2^{12}$ elements. The solid black line represents the deterministic solution obtained via the explicit Euler method with $\Delta t = 10^{-4}$. For the stochastic equations, the Euler–Maruyama method is applied using the same $\Delta t$.

In the left-hand panels, we display three simulations with distinct noise behaviors, while the right-hand panels compare two simulations with identical noise behavior but different $\mu$ values: one (blue, dotted) and two (red, dash-dotted), alongside the deterministic solution. The parameter $\mu$ significantly influences the deviation of the stochastic solution from the deterministic one, with further analysis on this effect provided in the subsequent subsection.

\begin{figure}[h!]
    \centering
        \includegraphics[width=0.49\textwidth]{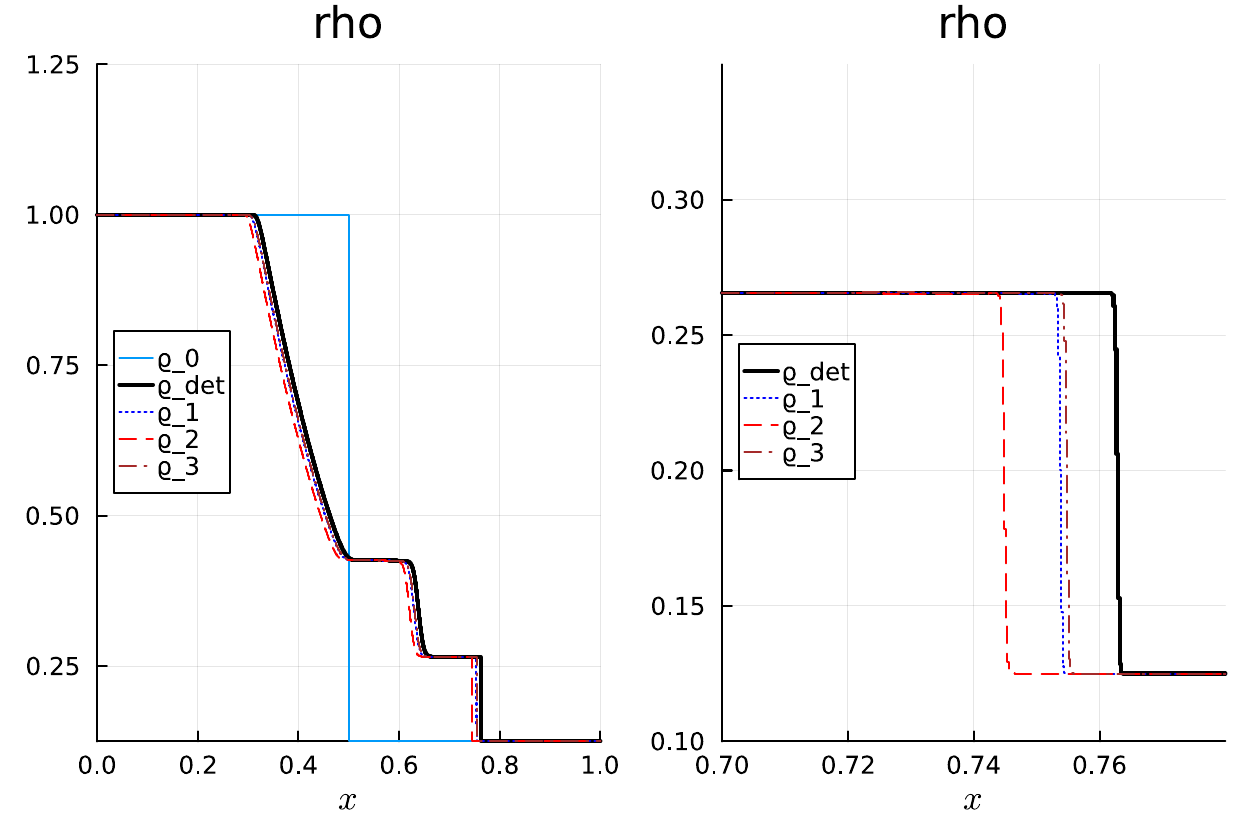}
        \includegraphics[width=0.49\textwidth]{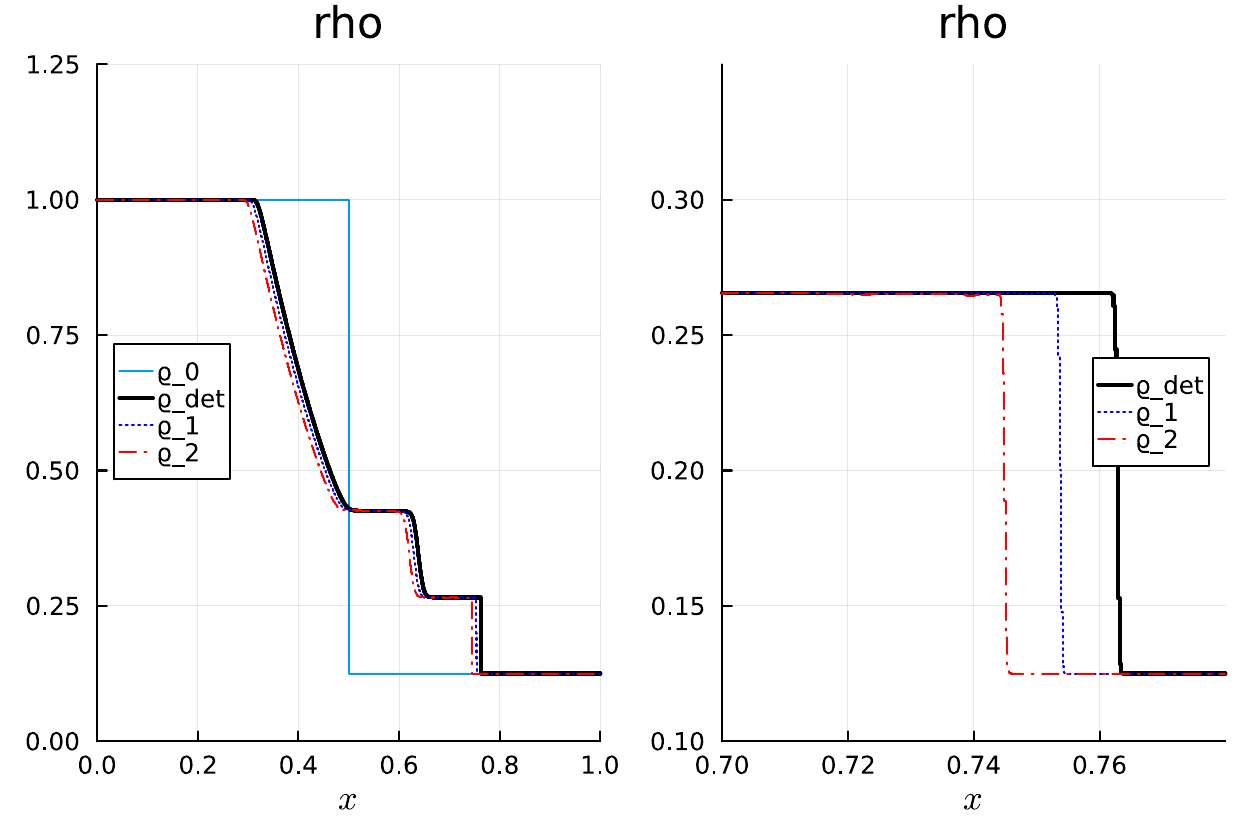}
    \caption{Different noise (left) and different noise strength(right)}
    \label{fig:SOD}
\end{figure}

%
%
%

\subsection{Two  dimension}\label{sec_two_d}

We consider  the two-dimensional stochastic compressible Euler equations \eqref{EulerC_short} with the right hand side 

\begin{equation*}
\mathbf{h}(\bU)=\mu^2\varrho\begin{pmatrix}0\\0\\ 0 \\1\end{pmatrix},\quad\mathbf{g}(\bU)\dd W= \mu\begin{pmatrix}0\\\varrho\, \dd \beta_1\\\varrho\,\dd \beta_2 \\m_1\,\dd\beta_1+m_2\,\dd\beta_2\end{pmatrix}.
\end{equation*}
where $\mu$ denotes again  a stochastic forcing factor. The domain is divided into $n\times n$ uniform quads for the elements ($n \in \N$). 

\subsubsection{Convergence  }
Similar to Section \ref{sub_sub_section_1},  we consider now the two dimensional density wave 
given by 
\begin{equation}
  \varrho(0, x) = 1 + 0.5 \sin\bigl(2\cdot x \bigr), \quad v_1(0, x) = 0.1, \quad v_2(0, x) = 0.1, \quad p(0,x)= 10,
\end{equation}
für $t \in [0, 0.1]$ and $x \in [-1, 1]^2$. Here, we use $\mu=0.1$ and calculate everything  with $\Delta t=10^{-4}$. The reference solution is calculated using $256^2$ elements. 
In Table \ref{FV_DG_density_2d}, we see the error behaviour for the FV and DGSEM method with $\mathfrak{p}=1$. We reach first and second order of accuracy which is in line with the
deterministic numerical simulations for a smooth test case, e.g. \cite{oeff2023,gassner2016,zbMATH07517718,chen2017} and references therein. 

%

%
%
%

{\small{
\begin{table}[h!]
\centering
\begin{tabular}{|c|c|c|c|c|c|c|c|c|}
\hline
& \multicolumn{4}{c|}{FV} & \multicolumn{4}{c|}{DGSEM ($\mathfrak{p}=1$)} \\
\cline{2-9}
Elements & $\mathtt{E}_1$ & EOC($\mathtt{E}_1$) & $\mathtt{E}_2$ & EOC($\mathtt{E}_2$) & $\mathtt{E}_1$ & EOC($\mathtt{E}_1$) & $\mathtt{E}_2$ & EOC($\mathtt{E}_2$) \\
\hline
$16^2$           &  0.2622  & -        &  11.5589 & -  &         0.0365   & -       &  0.6069   & -\\ 
$32^2$          &   0.1816  & 0.53  &   6.6302  & 0.80  &   0.0093   & 1.97 &  0.0453   &    3.74 \\
$64^2$          &   0.0974  & 0.90  &   2.4735 &  1.42  &  0.0022   & 2.05  &   0.0027  &  4.06  \\
$128^2$        &  0.0368   & 1.41   &   0.4531 &  2.45  &  0.0001   & 2.30  &   0.0001  &  4.59 \\ \hline
average EOC& -              & 0.95 &   -           &  1.52  & -              &  2.11 &  -            & 4.13 \\ \hline
\end{tabular}
\caption{LLF FV and  entropy stable DGSEM ($\mathfrak{p}=1$) methods - error plot using 1000 samples}
\label{FV_DG_density_2d}
\end{table}
}}

\subsubsection{Kelvin-Helmholtz and influence of the noise }
In the second test case in two-space dimensions, we consider the famous Kelvin-Helmholtz  instability test.  
Kelvin-Helmholtz describes a shear flow of three fluid layers with different densities. 
The initial data are given by 
\begin{equation}\label{init_KH}
 (\varrho, v_1, v_2, p)(\bx, 0)=\begin{cases}
                          (2,-0.5,0,2.5), \quad I_1 \leq y \leq I_2,\\
                          (1,0.5,0,2.5),   \text{ otherwise 0},
                         \end{cases}
\end{equation}
where the interface profiles $I_j=I_j(\bx):= J_j+\epsilon Y_j(\bx),\; j=1,2$ and $0<\epsilon\ll1$, are chosen to be small perturbations around the lower $J_1=0.25$ and the upper $J_2=0.75$ interfaces, respectively. Moreover, 
$Y_j= \sum_{m=1}^M a_j^m \cos \left( b_j^m +2\pi m x \right)$, $j=1,2$, with $a_j^m\in [0,1]$ and $b_j^m\in[-\pi, \pi],\; j=1,2,\; m=1,\cdots, M$ are arbitrary but fixed numbers. The coefficients $a_j^m$ have been normalized such that $\sum_{m=1}^M a_j^m=1$ to guarantee  that $|I_j-J_j|\leq \epsilon$ for $j=1,2.$ In the simulations, we have $M=10$, $\epsilon=0.01$ and $T=1.5$. We solve the Kelvin-Helmholtz problem on $[0,1]\times[0,1]$ with periodic boundary conditions. 
It is known that no grid conference can be obtained as can be seen for instance in \cite{Fjordholm2016, feireisl2020convergence,LuOe} and references therein. Vertices rise on different locations and fine scales which can also derive problems in terms of stability for entropy dissipative DG methods \cite{glaubitz2024generalized,ranocha2025robustness}.\\
However, to demonstrate now the differences driven by the noise and also by the forcing factor $\mu$, we print in Figure \ref{fig:main} different realisations using 
the flux differencing DGSEM method with $\mathfrak{p}=2$ on a grid with $64 \times 64$ elements. The time step is set to $\Delta t=10^{-4}$. We are using again explicit Euler in the deterministic case (for comparison) and the Euler–Maruyama in the stochastic setting. 
Subfigure \ref{fig:sub1} depicts the deterministic solution, whereas subfigures \ref{fig:sub2}–\ref{fig:sub3} show two different realizations of the same test case with distinct noise configurations. Subfigures \ref{fig:sub3}–\ref{fig:sub4} illustrate the same problem but with varying noise strengths. These results clearly highlight the significant impact of noise strength on the solution. 
We do not want to hide that if the strength is to high above $1$, we have obtained also stability issues in this test case for the tested realisations. This can possible also happen for other realisations with the smaller strength where the noise can have an stabilisation effect as well. A more detailed investigation of such context will left for future research. 

%

\begin{figure}[h!]
    \centering
    \begin{subfigure}{0.45\textwidth}
        \includegraphics[width=\linewidth]{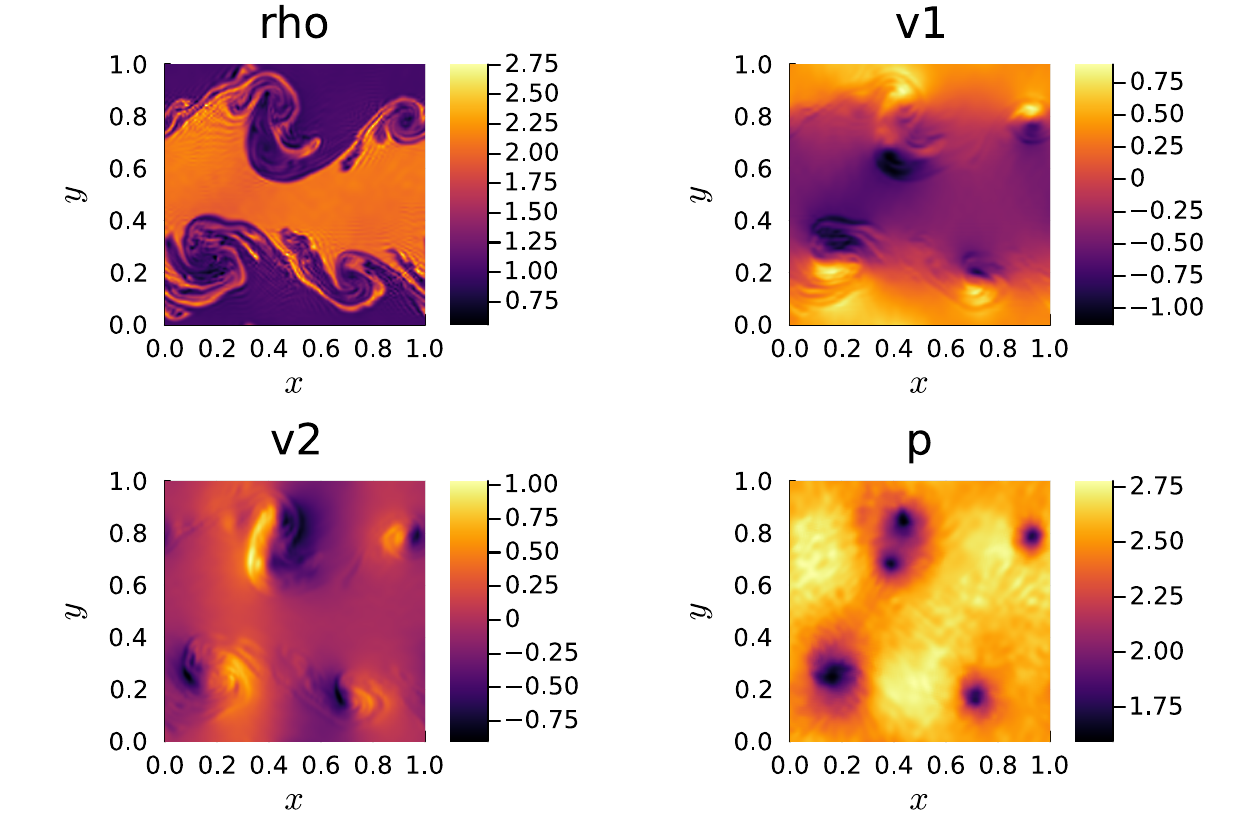}
        \caption{Deterministic}
        \label{fig:sub1}
    \end{subfigure}
    \begin{subfigure}{0.45\textwidth}
        \includegraphics[width=\linewidth]{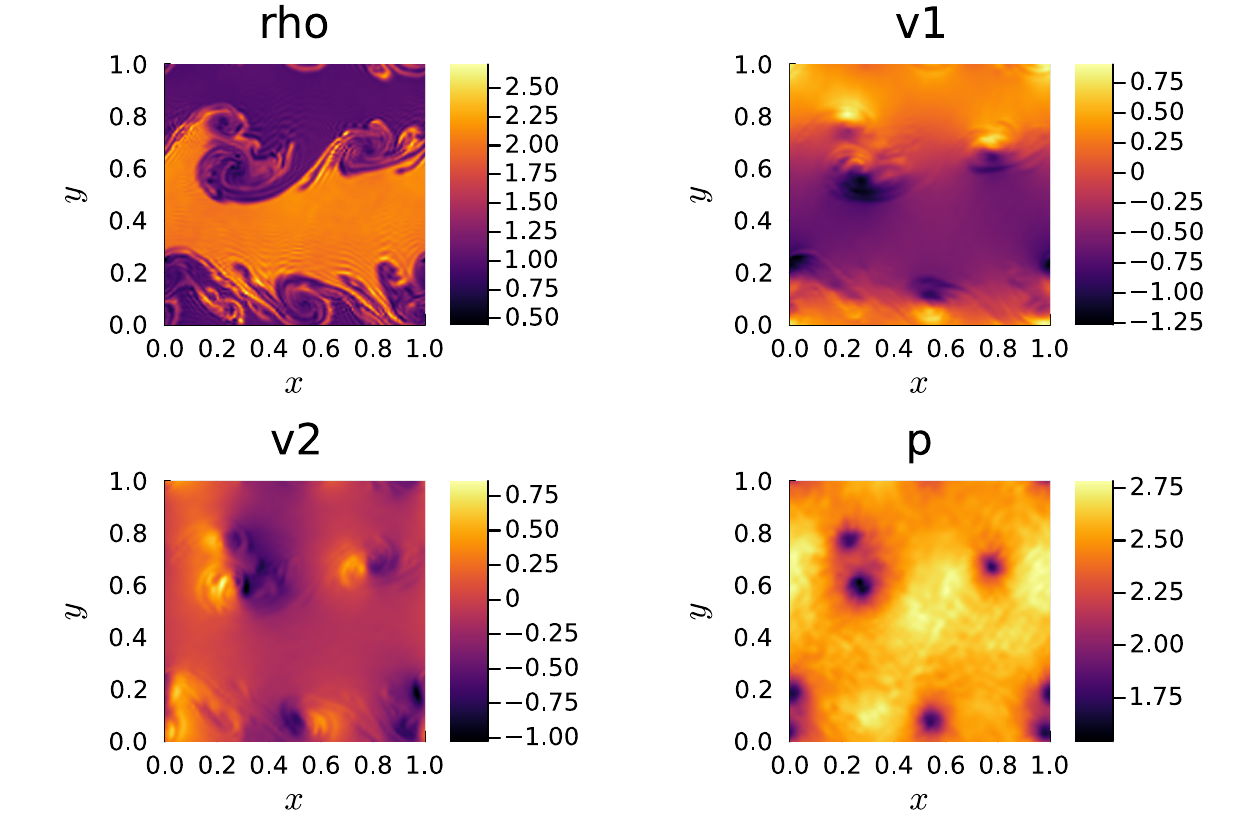}
        \caption{First realization with noise $\mu=0.1$}
        \label{fig:sub2}
    \end{subfigure}
    \begin{subfigure}{0.45\textwidth}
        \includegraphics[width=\linewidth]{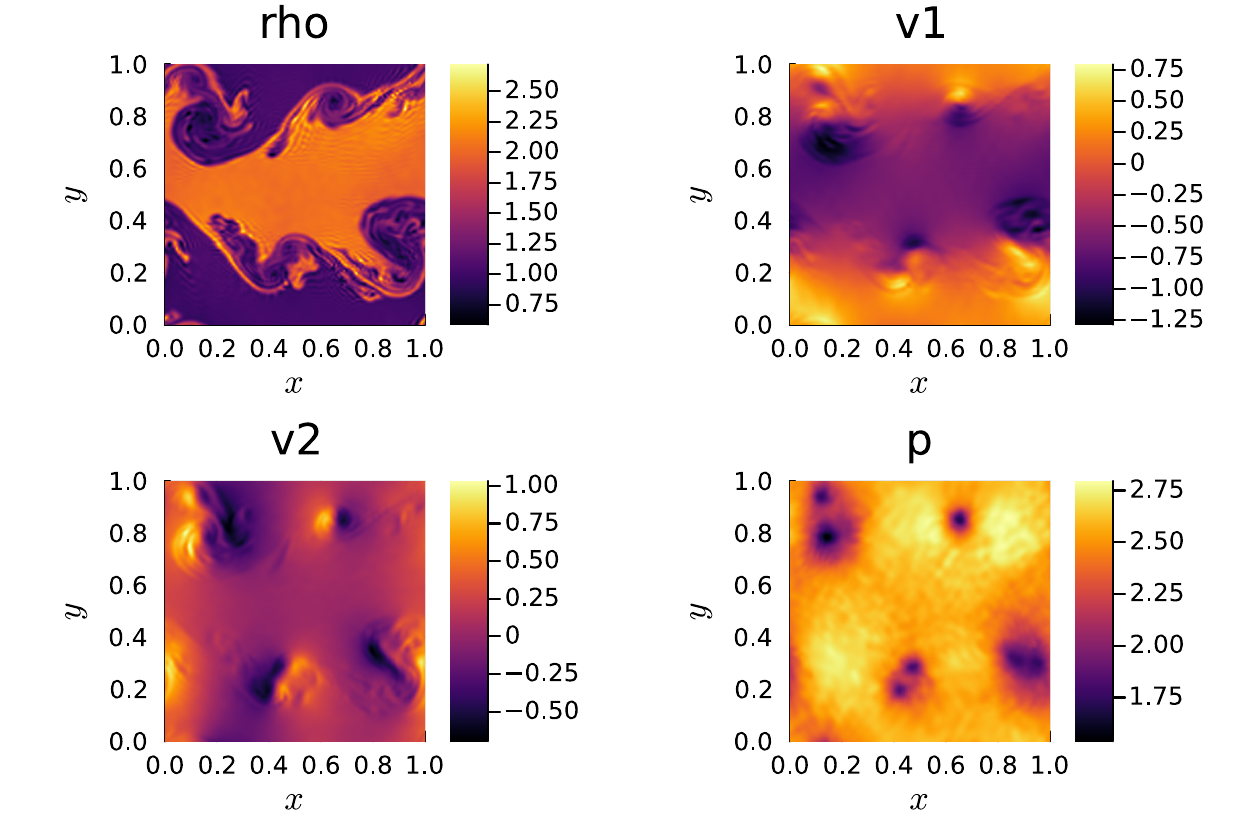}
        \caption{Second realization with noise $\mu=0.1$}
        \label{fig:sub3}
    \end{subfigure}
    \begin{subfigure}{0.45\textwidth}
        \includegraphics[width=\linewidth]{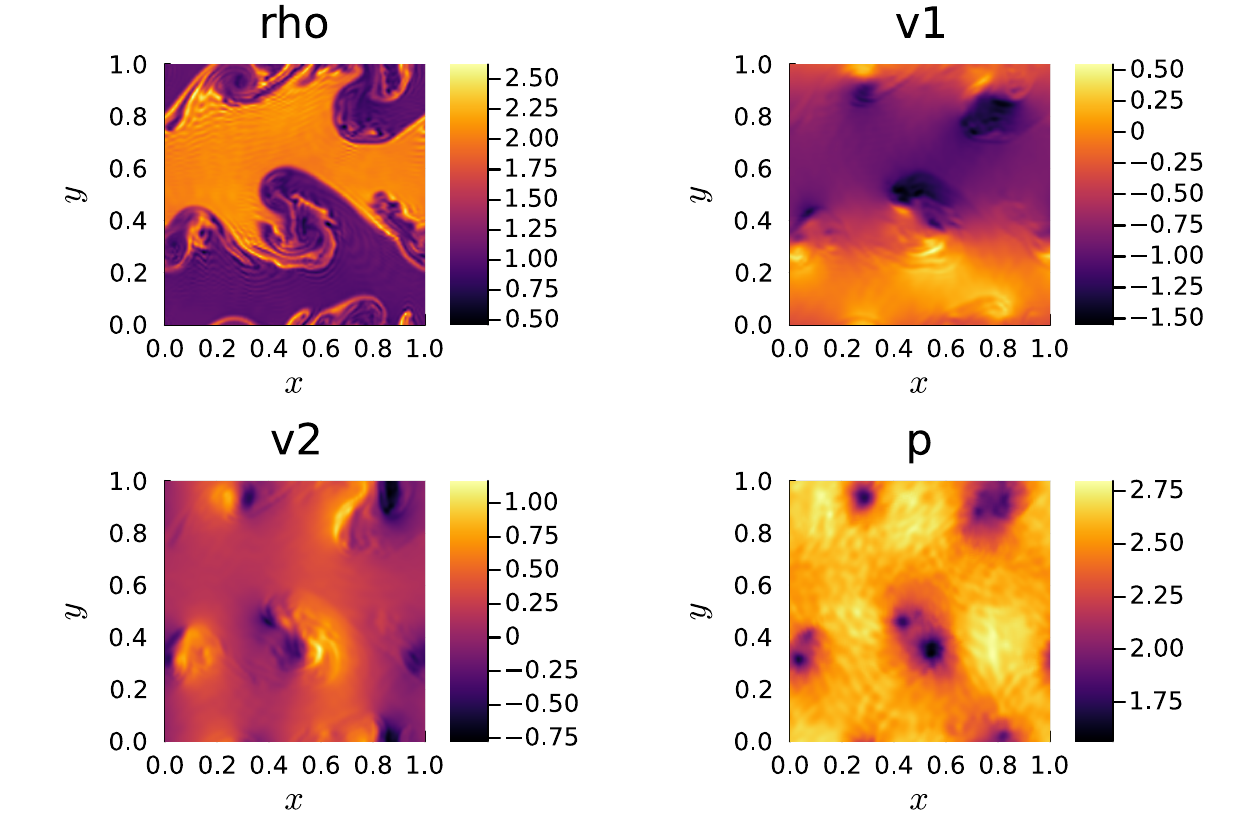}
        \caption{Second realization with noise $\mu=0.25$}
        \label{fig:sub4}
    \end{subfigure}
    \caption{Kelvin-Helmholtz and influence of the noise }
    \label{fig:main}
\end{figure}

\section{Conclusion and outlook}\label{se_outlock}

In this work, we have extended the entropy-dissipative DGSEM methods including the FV case to solve the stochastic Euler system, demonstrating convergence and providing error estimates across various settings. Our analysis builds on the consistency framework from \cite{LuOe} and the theoretical results in \cite{Mo}, leading to a first convergence result for the stochastic Euler equations. Additionally, we implemented these methods and examined error behavior across multiple experiments as a proof of concept.\\
We have seen that the method is stable and for sufficient smooth solutions (for the deterministic case), we obtain even better convergence rates than we would have expected. 
 This is also in line with the purely deterministic framework where several optimal error estimations for different DG settings are known as references for instance in \cite{zbMATH06698849,Jiao2022,Liu2020Liu,zbMATH07086311}. However, our analysis is relying on the consistency estimation where no regularity have been assumed. In the future, we plan to derive new error estimates using the consistency analysis and transfer it also from the  purely deterministic setting to the stochastic one.
Meanwhile, we have numerically analyzed the stochastic Euler equations using LLF-FV and  DGSEM methods. We observe that LLF-FV methods exhibit error behavior in the stochastic setting similar to that derived in \cite{LuSh} for the deterministic case. Furthermore, our testing of the DGSEM method shows that, with smooth initial data, numerical results resemble those of deterministic case. However, for cases involving discontinuities, we observe oscillations, indicating a need for the development of limiting strategies including stochasticity to address more complex test cases where a first step in this direction was recently done in \cite{Wo}. Future research will focus on these limiting strategies, along with turbulence modeling aspects.

\section*{Acknowledgement}

The work of D.B.
was supported by the German Research Foundation (DFG) within the framework of the priority research program SPP 2410 under the grant BR 4302/3-1 (525608987) and under the personal grant BR 4302/5-1 (543675748).\\
P.Ö. was supported by the DFG within SPP 2410, project OE 661/5-1 
(525866748) and under the personal grant  OE 661/4-1(520756621).\\
P.Ö. thanks Hendrik Ranocha (JGU Mainz) for his assistance in implementing the stochastic forcing term in Trixi.jl.

\section*{Conflict of Interest}
The authors have no conflict of interest to declare.

\bibliographystyle{abbrv}
\bibliography{literature}

\end{document}